\documentclass[a4paper,10pt,reqno]{amsart}
\usepackage{amsmath}%
\usepackage{amsthm}
\usepackage{amsfonts}%
\usepackage{amscd}
\usepackage{amssymb}%
\usepackage{graphicx}
\usepackage{enumerate}
\usepackage[all]{xy}
\usepackage{hyperref}
\usepackage{mathtools}
\usepackage{ stmaryrd }
\usepackage{xcolor}

\newtheorem{theorem}{Theorem}
\theoremstyle{plain}
\newtheorem{proposition}[theorem]{Proposition}
\newtheorem{lemma}[theorem]{Lemma}
\newtheorem{corollary}[theorem]{Corollary}

\theoremstyle{definition}
\newtheorem{definition}[theorem]{Definition}

\newtheorem{example}[theorem]{Example}

\theoremstyle{remark}

\newtheorem{remark}[theorem]{Remark}

\numberwithin{equation}{section}

\DeclareMathOperator{\tr}{tr}

\DeclareMathOperator{\Aut}{Aut}
\DeclareMathOperator{\Hom}{Hom}

\DeclareMathOperator{\C}{\mathcal{C}}
\DeclareMathOperator{\D}{\mathcal{D}}
\DeclareMathOperator{\E}{\mathcal{E}}

\DeclareMathOperator{\id}{id}
\DeclareMathOperator{\F}{\mathcal{F}}
\DeclareMathOperator{\G}{\mathcal{G}}

\newcommand{\N}{\mathbb{N}}

\begin{document}
	
\title[Separability for relative extensions]{Separability for relative extensions of object unital strongly groupoid graded rings}
	\author[Z. Cristiano]{Zaqueu Cristiano}
	\address[Zaqueu Cristiano]{Department of Mathematics - IME, University of S\~ao Paulo,
		Rua do Mat\~ao 1010, S\~ao Paulo, SP, 05508-090, Brazil}
	\email[Zaqueu Cristiano]{zaqueucristiano@usp.br}%
	
	\author[P. Lundstr\"{o}m]{Patrik Lundstr\"{o}m}
	\address[Patrik Lundstr\"{o}m]{University West,
Department of Engineering Science, 
SE-461 86 Trollh\"{a}ttan, Sweden}
	\email[Patrik Lundstr\"{o}m]{patrik.lundstrom@hv.se}%
	
	
	\date{\today}
	\subjclass[2020]{Primary . Secondary } %

\begin{abstract}
We prove that if $R$ is a ring that is
object unital and strongly graded by a
groupoid $\Gamma$, and if $\Delta$ is a wide
subgroupoid of $\Gamma$, then $R/R_\Delta$ is
separable if and only if, for each
$e \in \Gamma_0$, there exist $f \in [e]$ and $r \in C_{R_0}(R_\Lambda) :=
\{ x \in R_0 \mid xy = yx \text{ for all } y \in R_\Lambda \}$
with
\smash{${\rm tr}_{\Gamma/\Delta}^f(r) = 1_{R_f}$.}
Here, $\Gamma_0$ denotes the set of objects of $\Gamma$,
$[e]$ the connected component of $\Gamma_0$ containing $e$,
$\Lambda$ the isotropy groupoid of $\Delta$, and
\smash{${\rm tr}_{\Gamma/\Delta}^f$} 
the relative trace map at $f$.
This result simultaneously generalizes earlier theorems
on separability for matrix rings and group-graded rings due to
DeMeyer-Ingraham, N{\v a}st{\v a}sescu, Van den Bergh,
Van Oystaeyen, Miyashita, 
Theohari-Apostolidi, and Vavatsoulas,
as well as results on groupoid-graded 
rings due to
Cala, Lundstr\"{o}m, and Pinedo.
As an application, we consider separability for object crossed products,
including object twisted groupoid rings, 
classical groupoid rings
and matrix rings, as well as
crossed product algebras defined by
infinite separable field extensions.
\end{abstract}
	
\maketitle


\section{Introduction}\label{sec:introduction}

Let $R$ be a ring. By this we mean that 
$R$ is associative but not necessarily unital.
If $R$ is unital, we denote its multiplicative
unit by $1_R$ and assume $1_R \neq 0$.
Suppose that $S$ is a subring of $R$.
In that case, we say that $R$ is a 
ring extension of $S$, 
and we indicate this by writing $R/S$.

Recall that $R/S$ is said
to be \emph{separable} if the multiplication
map $\mu : R \otimes_S R \to R$,
defined by the additive extension of
$\mu(r \otimes r') := rr'$, 
for $r,r' \in R$,
admits a \emph{section} in the category of 
$R$-bimodules, that is, if
there is an $R$-bimodule map
$\delta : R \to R \otimes_S R$ such that
$\mu \circ \delta = \id_R$. 

The concept of separability for ring extensions
generalizes the classical notion of 
separability for algebras
over fields, which itself extends the notion
of separability for field extensions 
(see e.g. \cite{demeyer1971}). 
Separable ring extensions have in turn
been generalized through
the introduction of 
\emph{separable functors} in \cite{nastasescu1989}.
For a comprehensive overview of the historical 
development of the concept of separability, 
see \cite{wisbauer2016}.

Separable ring extensions have been studied
by numerous authors (see e.g. 
\cite{brzezinski2005,CLP,demeyer1971,
haefner2000,iglesias1998,kadison1999,
lundstrom2005,lundstrom2006,
miyashita1971,
nastasescu1989,nystedt2018,
rafael1990,TheohariApostolidi}).
The reason for the intense interest 
in such ring
extensions is that important properties 
of the ground ring, such as semisimplicity
and hereditarity, are often automatically 
inherited by the larger ring.

Conditions for separability of ring 
extensions have been studied in various settings.
Let us briefly describe some of these results.
The following results on
rings of matrices and group rings were
proved by DeMeyer and Ingraham in
\cite[p. 41]{demeyer1971}:

\begin{theorem}
\label{thm:Demeyer-Ingraham}
Suppose that $S$ is a unital ring.
\begin{enumerate}[{\rm (a)}]

\item Let $n \in \N := \{ 1,2,3,\ldots \}$
and let $R := M_n(S)$ denote the ring of 
$n \times n$ matrices over $S$. 
We consider $S$ as a subring of $R$
by identifying each element of $S$ with
its corresponding 
diagonal matrix. Then $R/S$ is separable.

\item Let $G$ be a finite group, $R:=S[G]$
the group ring of $G$ over $S$.
We consider $S$ as a subring of $R$
by identifying each 
$s \in S$ with $se \in R$,
where $e$ is the identity element of $G$.
If the order $|G|$ of $G$ is a unit in $S$, 
then $R/S$ is separable.

\end{enumerate}
\end{theorem}

Theorem \ref{thm:Demeyer-Ingraham}
has been generalized to 
various types of extensions of graded rings.
More precisely, let $G$ be a group.
Recall that a ring $R$ is said to be
\emph{graded by $G$} if 
there is a family $(R_g)_{g \in G}$ of
additive subgroups of $R$ satisfying
$R = \oplus_{g \in G} R_g$ and
$R_g R_h \subseteq R_{gh}$, for $g,h \in G$;
if also $R_g R_h = R_{gh}$, for $g,h \in G$,
then $R$ is said to be 
\emph{strongly} graded by $G$.

Suppose that $R$ is unital
and strongly graded by $G$. 
For each $g \in G$, choose $n_g \in \N$, 
\smash{$u_g^{(i)} \in R_g$} and
\smash{$v_{g^{-1}}^{(i)} \in R_{g^{-1}}$},
for $i \in \{ 1,\ldots,n_g\}$, such that
\smash{$1_R = \sum_{i=1}^{n_g} u_g^{(i)} 
v_{g^{-1}}^{(i)}$.} Define 
$\gamma_g : R \to R$ by 
$\gamma_g(r) = \sum_{i=1}^{n_g}  
u_g^{(i)} r v_{g^{-1}}^{(i)}$, 
for $r \in R$.
If $G$ is finite, we may define the \emph{trace map} 
$\tr : R \to R$ by $\tr(r) = \sum_{g \in G}
\gamma_g(r)$, for $r \in R$.
Let $Z(R_e) :=
\{ x \in R_e \mid xy = yx \text{ for all } y \in R_e \}$.
N{\v a}st{\v a}sescu, Van den Bergh and 
Van Oystaeyen proved the following 
result in \cite{nastasescu1989}:

\begin{theorem}\label{thm:nastasescu}
Let $R$ be a unital ring which is 
strongly graded by the group $G$.
Then $R/R_e$ is separable
if and only if $G$ is finite and 
there is $r \in Z(R_e)$ with $\tr(r) = 1_R$.
\end{theorem}

This result has been further generalized to
\emph{relative} ring extensions $R/R_H$, 
where $H$ is a subgroup
of $G$ and $R_H = \oplus_{h \in H} R_h$.
Namely, let $[G:H]$ denote the index of $H$ in $G$,
and choose a left transversal $T$ for $H$ 
in $G$. If $[G:H]$ is finite, then we 
define the \emph{relative trace map}
$\tr_{G/H} : R \to R$ by
$\tr_{G/H}(r) = \sum_{t \in T} \gamma_t(r)$,
for $r \in R$. 
The following result was proved by
Theohari-Apostolidi and Vavatsoulas in
\cite[Thm. 2.1]{TheohariApostolidi}:

\begin{theorem}\label{thm:miyashita}
Let $R$ be a unital ring which is 
strongly graded by the group $G$. Suppose that $H$ 
is a subgroup of $G$. 
Then $R/R_H$ is separable
if and only if $[G:H]$ is finite and 
there exists $r \in Z(R_H)$ such that 
\smash{$\tr_{G/H}(r) = 1_R$}.
\end{theorem}

Let us make some remarks on this result.
If $H = \{ e \}$, it is precisely 
Theorem \ref{thm:nastasescu}, and if 
$H$ is a normal subgroup of $G$, 
then it appears in 
\cite[Prop. 2.2]{haefner2000}. 
Note that the sufficient condition was
established already 
by Miyashita in \cite[Thm. 2.11]{miyashita1971}.

In this article, we generalize 
Theorem \ref{thm:miyashita} to the 
context of relative extensions of 
rings graded by \emph{groupoids}
(see Theorem \ref{thm:main}).
This class of rings includes
rings which are not, in any natural way,
graded by groups, such as 
matrices, crossed products 
\cite{lundstrom2005}
and groupoid rings \cite{lundstrom2006}.
Let us briefly describe these structures.

Let $\Gamma$ be a groupoid.
By this we mean that
$\Gamma$ is a small category in which 
every morphism is an isomorphism. 
We let $\Gamma_0$ denote the set 
of objects of $\Gamma$, identified with
their identity morphisms,
and we let $\Gamma_1$ denote the set of 
morphisms of $\Gamma$.
The \emph{domain} and \emph{range}
functions $\Gamma_1 \to \Gamma_0$
are denoted by $d$ and $r$, respectively.
Let $e,f \in \Gamma_0$. We let 
$\Gamma(e,f)$ denote the set of 
$\sigma \in \Gamma_1$ with 
$d(\sigma) = e$ and $r(\sigma) = f$;
this will often be indicated by writing
$\sigma : e \to f$. 
We identify each 
$e \in \Gamma_0$ with its corresponding 
identity 
morphism $e \to e$, thereby considering
$\Gamma_0$ as a subset of $\Gamma_1$.
From now on we write $\sigma \in \Gamma$
when we mean $\sigma \in \Gamma_1$.
We let the inverse of $\sigma \in \Gamma$
be denoted by $\sigma^{-1}$.
We let $\Gamma_2$ denote the set
of ordered pairs $(\sigma,\tau)$ in 
$\Gamma \times \Gamma$ with
$d(\sigma)=r(\tau)$.
If $(\sigma,\tau) \in \Gamma_2$, then the
composite of $\sigma$ and $\tau$ 
is denoted $\sigma\tau$.

By a \emph{subgroupoid} of $\Gamma$ we 
mean a subcategory $\Delta$ of $\Gamma$
which is a groupoid. In that case, we say
that $\Delta$ is \emph{wide} if 
$\Delta_0 = \Gamma_0$.
If $e \in \Gamma_0$, then $\Gamma(e,e)$ 
is a group called the 
\emph{isotropy group} at $e$;
we will denote this group $\Gamma(e)$.
The \emph{isotropy groupoid} 
${\rm Iso}(\Gamma)$ of 
$\Gamma$ is the groupoid 
having $\Gamma_0$ as objects and
$\cup_{e \in \Gamma_0} \Gamma(e)$
as morphisms.

Let $R$ be a ring. Following 
\cite{lundstrom2005,lundstrom2006}, 
we say that $R$ is 
\emph{graded by $\Gamma$} if there
is a family 
$(R_\sigma)_{\sigma \in \Gamma}$
of additive subgroups of $R$ 
with 
$R = \oplus_{\sigma \in \Gamma} R_\sigma$,
and for all
$\sigma,\tau \in \Gamma$,
$R_\sigma R_\tau \subseteq R_{\sigma\tau}$,
if $(\sigma,\tau) \in \Gamma_2$, and 
$R_\sigma R_\tau = \{ 0 \}$, if
$(\sigma,\tau) \notin \Gamma_2$.
In that case, $R$ is said to be 
\emph{strongly} graded by $\Gamma$ if 
$R_\sigma R_\tau = R_{\sigma\tau}$,
for $(\sigma,\tau) \in \Gamma_2$.

Following \cite{CLP,cala2022}, 
we say that a $\Gamma$-grading on $R$ is 
\emph{object unital}
if for each $e \in \Gamma_0$, 
the ring $R_e$ 
is unital, and for each
$\sigma \in \Gamma$ 
and each $r \in R_\sigma$, 
$1_{R_{r(\sigma)}} r =
r 1_{R_{d(\sigma)}} = r$.
For a subset $\Delta$ of $\Gamma$
we put $R_\Delta := 
\oplus_{\delta \in \Delta} R_\delta$.
We also set $R_0:=R_{\Gamma_0}$.

In \cite{CLP}, a generalization of 
Theorem \ref{thm:nastasescu} is
established for ring extensions
$R/R_0$ where $R$ is a ring graded 
by a groupoid $\Gamma$
(see Theorem \ref{thm:CLP}).
Suppose that $R$ is strongly  
graded by $\Gamma$ and, with respect
to this grading, object unital.
Then, for each $e \in \Gamma_0$, 
$R_{\Gamma(e)}$ is a unital ring
which is strongly graded by the group 
$\Gamma(e)$ and we can define the 
trace map ${\rm tr}^e : R_{\Gamma(e)}
\to R_{\Gamma(e)}$ as in the 
discussion preceding 
Theorem \ref{thm:nastasescu}.
In \cite[Thm. 2]{CLP} the following result 
is proved.

\begin{theorem}\label{thm:CLP}
Let $R$ be a ring which is strongly 
graded by the groupoid $\Gamma$.
Suppose that with respect to this 
grading, $R$ is object unital.
Then $R/R_0$ is separable if and only if
for each $e \in \Gamma_0$,
$\Gamma(e)$ is finite and there is 
$r \in Z( R_e )$ with
${\rm tr}^e(r) = 1_{R_e}$.
\end{theorem}

In this article, we prove a
simultaneous generalization
of Theorems 
\ref{thm:Demeyer-Ingraham}-\ref{thm:CLP} in the context
of \emph{relative} extensions 
$R/R_\Delta$ of
groupoid graded rings
(see Theorem \ref{thm:main}). 

Suppose that $R$ is strongly  
graded by $\Gamma$ and, with respect
to this grading, object unital.
Take $\sigma \in \Gamma$.
Choose $n_\sigma \in \N$,
and $u_\sigma^{(i)} \in R_\sigma$ and $v_{\sigma^{-1}}^{(i)} \in R_{\sigma^{-1}}$, 
for $i \in \{ 1,\ldots,n_\sigma \}$,
such that $1_{r(\sigma)} = 
\sum_{i=1}^{n_\sigma} u_\sigma^{(i)}
v_{\sigma^{-1}}^{(i)}$.
Define $\gamma_\sigma : 
R \to R$ by setting $\gamma_\sigma(r) =  
\sum_{i=1}^{n_\sigma} u_\sigma^{(i)} r
v_{\sigma^{-1}}^{(i)}$, for $r \in R$.
Take $\sigma,\tau \in \Gamma$.

Let $\Delta$ be a wide subgroupoid of 
$\Gamma$ and set
$\Lambda := {\rm Iso}(\Delta)$.
We say that 
$\sigma$ and $\tau$ are \emph{right $\Lambda$-equivalent} 
if $\sigma \Lambda = \tau \Lambda$;
right $\Lambda$-equivalence is clearly an 
equivalence relation on $\Gamma$.
Let $T$ denote a set of representatives
for the different equivalence classes 
for this relation. 
For $e,f \in \Gamma_0$,
put $T_{f,e} := \{ \sigma \in T \mid
\mbox{$d(\sigma)=f$ and $r(\sigma)=e$} \}$
and $[e] := \{ f \in \Gamma_0 \mid
\Gamma(e,f) \neq \emptyset \}$.
We set $\Gamma_0^{\rm fin} :=
\{ f \in \Gamma_0 \mid [\Gamma(f):\Delta(f)] < \infty \}$ and
we define the relative trace map 
${\rm tr}_{\Gamma/\Delta}^e : R \to R$ by 
${\rm tr}_{\Gamma/\Delta}^e(r) = 
\sum_{f \in \Gamma_0^{\rm fin}} 
\sum_{\tau \in T_{f,e}} \gamma_\tau(r)$,
for $r \in R$. For $X,Y \subseteq R$, set 
$C_X(Y) := \{ x \in X \mid xy = yx \ \mbox{for all 
$y \in Y$} \}$.

In 
Section~\ref{sec:separabilitygeneralcase},
we prove the following result.

\begin{theorem}\label{thm:main}
Let $R$ be a ring which is strongly 
graded by the groupoid $\Gamma$.
Suppose that with respect to this 
grading, $R$ is object unital.
Let $\Delta$ be a wide subgroupoid 
of $\Gamma$ and put 
$\Lambda := {\rm Iso}(\Delta)$.
Then $R/R_\Delta$ is 
separable if and only if for each 
$e \in \Gamma_0$, there exist 
$f \in [e]$ and
$r \in C_{R_0}(R_\Lambda)$ satisfying
${\rm tr}_{\Gamma/\Delta}^f(r) = 1_{R_f}$.
\end{theorem}

Theorem \ref{thm:main} is a 
simultaneous generalization of 
Theorems 
\ref{thm:nastasescu},
\ref{thm:miyashita} and \ref{thm:CLP}.
Indeed:
\begin{itemize}

\item Specializing to 
$\Gamma = G$ a group and 
$\Delta = \{ e \}$, 
we obtain Theorem \ref{thm:nastasescu}.

\item Specializing to $\Gamma = G$ 
a group and $\Delta = H$
a subgroup of $G$, we 
obtain a \emph{sharpening} of 
Theorem \ref{thm:miyashita}, replacing
$Z(R_H)$ by the smaller subring 
$C_{R_e}(R_H)$ (see Corollary
\ref{cor:miyashita}).
In fact, there appears 
to be a gap in the proof of
\cite[Thm.~2.1]{TheohariApostolidi}.
Our proof of Theorem \ref{thm:main}
fills this gap (see Remark \ref{rem:gap}).

\item Specializing 
to $\Delta = \Gamma_0$,
we retrieve Theorem \ref{thm:CLP}.

\end{itemize}

Here is an outline of the article. 

In Section \ref{sec:preliminaries}, we fix our notation for categories and we recall some notions on separable functors, along with conventions for rings and modules, including unitary modules and their restriction and induction functors. 
In particular, we state a key result
(see Theorem~\ref{thm:sepidempotent}) 
from \cite{CLP}
on separability of certain functors 
for module categories 
over rings with enough idempotents 
which we will need in later sections. 

In Section 
\ref{sec:separabilitygeneralcase},
we develop the machinery needed to 
prove Theorem \ref{thm:main}. 

In Section 
\ref{sec:separabilitynormalcase},
we specialize Theorem \ref{thm:main}
to the case when $\Delta$ is a normal
subgroupoid of $\Gamma$, in particular
the case when $\Delta = \Gamma_0$.

In Section 
\ref{sec:objectcrossedproducts}, 
we apply Theorem \ref{thm:main} 
to object crossed products, 
including object twisted groupoid rings, 
classical groupoid rings
and matrix rings, as well as
crossed product algebras defined by
infinite separable field extensions.


\section{Preliminaries}\label{sec:preliminaries}

In this section, we state our conventions
and some useful results
on categories, separable functors, 
rings, modules and separable 
extensions of rings. 

\subsection{Separable functors}

We use the notation and conventions for categories from
Section~\ref{sec:introduction}, with groupoids replaced by arbitrary
categories. Thus, if $\C$ is a category, then $\C_0$ and $\C_1$
denote the classes of objects and morphisms of $\C$, respectively;
we write $\sigma \in \C$ to mean that $\sigma \in \C_1$,
denote the domain and range maps by $d$ and $r$,
write $\C(e,f)$ for the class of morphisms $e \to f$,
identify each object with its identity morphism,
and denote composition by juxtaposition.

Let $\C$ and $\D$ be categories, and let $\F : \C \to \D$ be a functor
(always assumed to be \emph{covariant}).
Following \cite{nastasescu1989}, we say that $\F$ is \emph{separable}
if for all $e,f \in \C_0$, there is a map
$\varphi_{e,f}^{\F} : \Hom_{\D}(\F(e),\F(f)) \to \Hom_{\C}(e,f)$
satisfying:

\begin{itemize}

\item[{\rm (SF1)}] 
$\varphi_{e,e'}^{\F}
( \F(\alpha) ) = \alpha$, and 

\item[{\rm (SF2)}] 
$\F(\beta) \sigma = 
\tau \F(\alpha)$ $\Rightarrow$ 
$\beta \varphi_{e,f}^{\F}(\sigma) = 
\varphi_{e',f'}^{\F}(\tau) \alpha$,

\end{itemize}
for $e,e',f,f' \in \C_0$,
$\alpha \in \Hom_{\C}(e,e')$, 
$\beta \in \Hom_{\C}(f,f')$,
$\sigma \in \Hom_{\D}(\F(e),\F(f))$, 
and $\tau \in \Hom_{\D}(\F(e'),\F(f'))$.
For future use, we record the following:

\begin{proposition}\label{prop:nastasescu}
Let $\F : \C \to \D$ and $\G : \D \to \E$
be functors.

\begin{enumerate}[{\rm (a)}]

\item 
If $\F$ and $\G$ are separable, 
then $\G \circ \F$ is separable. 

\item
If $\G \circ \F$ is separable, 
then $\F$ is separable.

\end{enumerate}
\end{proposition}

\begin{proof}
See \cite[Lemma 1.1]{nastasescu1989}.
\end{proof}

\subsection{Modules}

Let $A$ be a ring.
By a left $A$-module we mean an additive group $M$
equipped with a biadditive map $A \times M \ni 
(a,m) \mapsto am \in M$ satisfying
$a ( b m ) = (a b) m$ for $a,b \in A$ 
and $m \in M$. Analogously, right $A$-modules
are defined.
We let ${}_A {\rm Mod}$ 
denote the category having 
left $A$-modules 
as objects and left $A$-module
homomorphisms as morphisms. 
Let $B$ be another ring. If $M$ is both a left 
$A$-module and a right $B$-module, 
then we say that $M$ is an $A$-$B$-bimodule if 
$(am)b = a(mb)$ for $a \in A$, $b \in B$ and $m \in M$. 
Recall that if $N$ is another $A$-$B$-bimodule, then a map
$f : M \to N$ is said to be an $A$-$B$-bimodule
homomorphism if it is simultaneously a left
$A$-module homomorphism and a right 
$B$-module homomorphism.

Let $M$ be a left $A$-module.
If $X \subseteq A$ and $Y \subseteq M$,
we let $XY$ denote the set of finite sums of
elements of the form $xy$, for $x \in X$ and $y \in Y$.
We say that $M$ is \emph{unitary} if $AM = M$.
We let ${}_A {\rm UMod}$ 
denote the full subcategory of
${}_A {\rm Mod}$ 
having unitary $A$-modules as objects.
Recall that $A$ is said to be \emph{idempotent} 
if $A$ is unitary as a left module over itself, 
that is if $AA = A$.

Suppose that $A/B$ is a ring extension 
where $A$ is idempotent.
To this ring extension we associate the
\begin{itemize}

\item \emph{restriction} functor 
${\rm Res}_{A/B} : 
{}_A {\rm UMod} \to 
{}_B {\rm Mod}$ 
which to a left $A$-module $M$ 
associates the left $B$-module $M$, 
where $B$ acts on $M$ via $A$, and the
 
\item \emph{induction} functor 
${\rm Ind}_{A/B} : 
{}_B {\rm UMod} \to {}_A {\rm UMod}$ 
which to a left $B$-module $N$
associates the left $A$-module 
$A \otimes_B N$,
where $A$ acts on $A \otimes_B N$ by left
multiplication, and
$B$ acts on $A$ from the right via $A$.

\end{itemize}

\begin{remark}
(a) Since $A$ is idempotent, 
${\rm Res}_{A/B}$ 
can be applied to $A$ itself.

(b) The functor ${\rm Res}_{A/B}$ 
does not always map
unitary modules to unitary modules. Indeed,
let $F$ be a field and put $A = F \times F$,
$B = \{ 0 \} \times F$ and $M = F \times \{ 0 \}$.
Let $A$ and $B$ act on $M$ by 
coordinatewise
multiplication. Then $AM = M$, but 
$BM = \{ (0,0) \} \neq M$.
Therefore, $M$ is unitary considered as a 
left 
module over $A$, but not over $B$.

(c) The functor ${\rm Ind}_{A/B}$ 
is well defined,
since $A(A \otimes_B N) = 
(AA) \otimes_B N = A \otimes_B N$.
\end{remark}

In \cite[Prop. 1.3]{nastasescu1989} the
following characterization of separability 
for the restriction and induction 
functors was obtained for unital rings:

\begin{theorem}\label{thm:oldsep}
Suppose that $A/B$ is an extension
of unital rings such that $1_A=1_B$.

\begin{enumerate}[{\rm (a)}]

\item ${\rm Res}_{A/B}$ 
is separable
if and only if $A/B$ is separable,
if and only if there exists 
$x \in A \otimes_B A$
such that $\mu(x) = 1_A$ and
$xa = ax$ for all $a \in A$.

\item ${\rm Ind}_{A/B}$ 
is separable
if and only if $B$ is a 
$B$-bimodule direct summand of $A$.

\end{enumerate}
\end{theorem}

Recall that two idempotents $u$ and $v$
in $A$ are called \emph{orthogonal} if $uv=0$.
Following Fuller \cite{fuller1976}, 
we say that $A$ has 
\emph{enough idempotents} if there is a 
nonemtpy set $U$ of nonzero orthogonal 
idempotents in $A$
(called a complete set of idempotents for $A$)
such that $A = \oplus_{u \in U} Au = 
\oplus_{u \in U} uA$. In that case, clearly, 
$A$ is idempotent.

\begin{lemma}\label{lem:um=m}
Let $A$ be a ring with enough idempotents 
having $U$ as a complete set of idempotents.
Suppose that $M$ is a unitary left $A$-module.
Take $m \in M$. Then $um = 0$ for all but finitely 
many $u \in U$ and $\sum_{u \in U} u m = m$.
\end{lemma}

\begin{proof}
Since $AM=M$ there are 
$n \in \N$, $a_1,\ldots,a_n \in A$ and
$m_1,\ldots,m_n \in M$ with $m = \sum_{i=1}^n a_i m_i$.
From the assumptions it follows that 
there is a finite subset $V$ of $U$ such that
for each $i \in \{ 1,\ldots,n \}$, and each 
$u \in U \setminus V$, 
$\sum_{v \in V} v a_i = a_i$ and 
$ua_i = 0$. Hence, for each
$u \in U \setminus V$,
$um = \sum_{i=1}^n u a_i m_i = 0$, and thus 
$\sum_{u \in U} um = \sum_{v \in V} v m +
\sum_{u \in U \setminus V} um = 
\sum_{i=1}^n \sum_{v \in V} v a_i m_i + 0 =
\sum_{i=1}^n a_i m_i = m$.
\end{proof}

\begin{proposition}
\label{prop:umod-->umod}
Let $A/B$ be an extension of
rings with enough idempotents having 
a common complete set of idempotents.
Then ${\rm Res}_{A/B} : 
{}_A {\rm UMod} \to {}_B {\rm UMod}$.
\end{proposition}

\begin{proof}
This follows immediately from 
Lemma \ref{lem:um=m}.
\end{proof}

In \cite[Prop.~26 and Prop.~27]{CLP}, the following generalization of Theorem~\ref{thm:oldsep} was established for rings with enough idempotents.

\begin{theorem}\label{thm:sepidempotent}
Let $A/B$ be an extension of
rings with enough idempotents having
a common complete set $U$ of idempotents.
\begin{enumerate}[{\rm (a)}]

\item ${\rm Res}_{A/B}$
is separable if and only if
$A/B$ is separable,
if and only if
for each $u \in U$ there exists
$x_u \in \sum_{v \in U} uAv \otimes vAu$
such that, for every $v \in U$
and every $a \in uAv$, the equalities
$\mu(x_u) = u$ and $x_u a = a x_v$ hold.

\item ${\rm Ind}_{A/B}$
is separable if and only if
$B$ is a
$B$-bimodule direct summand of $A$,
if and only if
for all $v,w \in U$, $vBw$ is a
$vBv$-$wBw$-bimodule direct summand of
$vAw$.
\end{enumerate}
\end{theorem}

In \cite{lundstrom2026}, a version of
Theorem \ref{thm:sepidempotent} was
proved in the context of firm modules.

\begin{corollary}
\label{cor: C-B-A}
Let $A/B$ and $B/C$ be extensions of
rings with enough idempotents having
a common complete set of idempotents.
\begin{enumerate}[\rm (a)]

\item If $A/B$ and $B/C$ are separable, then $A/C$ is separable.

\item If $A/C$ is separable, then $A/B$ is separable.

\end{enumerate}
\end{corollary}

\begin{proof}
By Theorem~\ref{thm:sepidempotent}(a),
separability of $A/B$ and $B/C$ is 
equivalent to separability of 
the functors
${\rm Res}_{A/B} : 
{}_A {\rm UMod} \to {}_B {\rm UMod}$ 
and 
${\rm Res}_{B/C}:
{}_B {\rm UMod} \to {}_C {\rm UMod}$. 
Since 
${\rm Res}_{B/C} \circ  
{\rm Res}_{A/B} =
{\rm Res}_{A/C}$, (a) and
(b) follow from 
Proposition
\ref{prop:nastasescu}.
\end{proof}

\begin{proposition}\label{prop:partition}
Let $A/B$ be an extension of
rings with enough idempotents having 
a common complete set $U$ of idempotents.
Suppose that there is a partition
$(U_i)_{i \in I}$ of $U$ such that
$A = \oplus_{i \in I} (U_i A U_i)$ and 
$B = \oplus_{i \in I} (U_i B U_i)$.
Then $A/B$ is separable if and only 
if for each $i \in I$, the ring extension
$(U_i A U_i)/(U_i B U_i)$ is separable.
\end{proposition}

\begin{proof}
Set $A_i := U_i A U_i$ and $B_i := U_i B U_i$ for each $i \in I$. Then
$A = \oplus_{i \in I} A_i$,
$B = \oplus_{i \in I} B_i$,
and $A_i A_j = 0$ for $i \neq j$.
Also, $A_i \otimes_B A_j = 0$ for 
$i \neq j$. Indeed, let $x \in A_i$ and 
$y \in A_j$. 
Since $x$ is a finite sum of terms 
$uav$ with $u,v \in U_i$,
the sum of the finitely many right idempotents $v$ occurring is an 
idempotent $e \in B_i$ satisfying $xe = x$. 
As $i \neq j$, we have $ey = 0$, and hence
$x \otimes y = xe \otimes y = 
x \otimes ey = 
x \otimes 0 = 0$ in $A \otimes_B A$.
Next, for each $i \in I$, define the map
$\varphi_i : A_i \otimes_B A_i 
\to A_i \otimes_{B_i} A_i$, by
$\varphi_i(x \otimes y) = x \otimes y$.
To see that $\varphi_i$ is well defined, let $x,y \in A_i$ and $b = \sum_{j \in I} b_j \in B$ with $b_j \in B_j$. Since $A_i B_j = 0 = B_j A_i$ for $j \neq i$, we have $xb = x b_i$ and $by = b_i y$, and thus
$xb \otimes y = x b_i \otimes y = x \otimes b_i y = x \otimes by$
in $A_i \otimes_{B_i} A_i$. Hence $\varphi_i$ is well defined. It is clearly $A_i$-bilinear, and its inverse is given by the same formula, so $\varphi_i$ is an isomorphism of $A_i$-bimodules.
Hence, there is an isomorphism of 
$A$-bimodules
$A \otimes_B A \cong 
\oplus_{i \in I} 
\left( A_i \otimes_{B_i} A_i \right)$,
where the $A$-bimodule structure on 
each summand is induced via the inclusion 
$A_i \subseteq A$. Under this 
identification, the multiplication map
$\mu : A \otimes_B A \to A$
corresponds to the direct sum of the multiplication maps
$\mu_i : A_i \otimes_{B_i} A_i \to A_i$.

Suppose that $A/B$ is separable. Then there exists an $A$-bimodule homomorphism $\sigma : A \to A \otimes_B A$ such that $\mu \circ \sigma = \mathrm{id}_A$. Composing $\sigma$ with the projection onto the $i$-th summand yields an $A_i$-bimodule homomorphism
$\sigma_i : 
A_i \to A_i \otimes_{B_i} A_i$
satisfying $\mu_i \circ \sigma_i = \mathrm{id}_{A_i}$. Thus each $A_i/B_i$ is separable.

Conversely, suppose that each $A_i/B_i$ is separable. For each $i \in I$, suppose that
$\sigma_i : A_i \to A_i \otimes_{B_i} A_i$ 
is an $A_i$-bimodule homomorphism with
$\mu_i \circ \sigma_i = \mathrm{id}_{A_i}$. Define
$\sigma : A \to A \otimes_B A$
by
$\sigma ( \sum_{i \in I} x_i) := 
\sum_{i \in I} \sigma_i(x_i)$,
where $x_i \in A_i$ and only finitely many $x_i$ are nonzero. Then $\sigma$ is a well-defined $A$-bimodule homomorphism, and $\mu \circ \sigma = \mathrm{id}_A$. Hence $A/B$ is separable.
\end{proof}

\section{Separability: general case}\label{sec:separabilitygeneralcase}

In this section, we prove
Theorem \ref{thm:main}
(see Theorem 
\ref{prop: separability R over R_Delta}).
Throughout this section, we keep the 
notation used in 
Section \ref{sec:introduction} and
we make the following assumptions:

\begin{itemize}

\item the ring $R$ is graded by
$\Gamma$, and this grading is object unital;

\item we put $1_e := 1_{R_e}$,
for $e \in \Gamma_0$, and 
$R_0 := R_{\Gamma_0}$;

\item $\Delta$ is 
a wide subgroupoid of $\Gamma$ and
we set $\Lambda := {\rm Iso}(\Delta)
= \cup_{e \in \Gamma_0} \Delta(e)$.

\end{itemize}

Note that since $\Delta$ is a wide 
subgroupoid of $\Gamma$, 
$R$ and $R_\Delta$ are rings 
with enough idempotents having a 
common complete set of idempotents,
namely $\{ 1_e \}_{e \in \Gamma_0}$.

\begin{proposition}
${\rm Ind}_{R/{R_\Delta}} :
{}_{R_\Delta} {\rm UMod} \to 
{}_R {\rm UMod}$ is separable.
\end{proposition}

\begin{proof}
Take $e,f \in \Delta_0$. 
Then
$1_e R_{\Delta} 1_e = R_{\Delta(e)}$, 
$1_f R_{\Delta} 1_f = R_{\Delta(f)}$,  
and $1_e R_{\Delta} 1_f = R_{\Delta(f,e)}$.
Also
$1_e R 1_f = R_{\Delta(f,e)} \oplus 
R_{\Gamma(f,e) \setminus \Delta(f,e)}$
as additive groups. Clearly, 
$R_{\Delta(f,e)}$ is an 
$R_{\Delta(e)}$-$R_{\Delta(f)}$-bimodule.
Since $\Gamma$ is a groupoid and 
$\Delta$ is a subgroupoid of $\Gamma$,
$\Delta(e) \big( \Gamma(f,e) 
\setminus \Delta(f,e) \big)
\Delta(f) \subseteq \Gamma(f,e) 
\setminus \Delta(f,e)$.
Thus,
$R_{\Delta(f,e)}$ is an
$R_{\Delta(e)}$-$R_{\Delta(f)}$-bimodule
direct summand of $1_eR1_f$.
By Theorem~\ref{thm:sepidempotent}(b),
${\rm Ind}_{R/R_\Delta}$ is separable.
\end{proof}

We now turn to the question of separability 
of 
${\rm Res}_{R/{R_\Delta}} :
{}_R {\rm UMod} \to 
{}_{R_\Delta} {\rm UMod}$. 

\begin{definition}
\label{def: connected components}
Take $e,f \in \Gamma_0$.
Define an equivalence relation $\sim$ on 
$\Gamma_0$ by saying that $e\sim f$ if 
$\Gamma(e,f)\neq\emptyset$. 
We let $[e]$ denote
the equivalence class containing $e$. 
We say that $\Gamma$ is 
\emph{connected} if there is only one
equivalence class defined by $\sim$.
\end{definition}

By the next result, it is enough 
to consider the case 
when $\Gamma$ is connected.

\begin{proposition}\label{prop:connected}
For $e \in \Gamma_0$, 
put $\Gamma_e :=
\{ \sigma \in \Gamma \mid 
d(\sigma),r(\sigma) \in [e] \}$ and
$\Delta_e := \Delta \cap \Gamma_e$.
For each $e \in \Gamma_0$, 
$\Gamma_e$ is connected, and
the following assertions are equivalent:
 \begin{enumerate}[\rm (a)]

\item ${\rm Res}_{R/{R_\Delta}} :
{}_R {\rm UMod} \to 
{}_{R_\Delta} {\rm UMod}$
is separable.

\item The ring extension $R/R_\Delta$ 
is separable.

\item For each $e \in \Gamma_0$, the ring 
extension $R_{\Gamma_e}/R_{\Delta_e}$
is separable.

\item For each $e \in \Gamma_0$, 
${\rm Res}_{R_{\Gamma_e}/R_{\Delta_e}} :
{}_{R_{\Gamma_e}} {\rm UMod} \to 
{}_{R_{\Delta_e}} {\rm UMod}$
is separable.

\end{enumerate}
\end{proposition}

\begin{proof}
Put $U := \{ 1_e \mid e \in \Gamma_0 \}$.
Let $\{ e_i \}_{i \in I}$
be a set of representatives for the
different equivalence classes defined
by $\sim$. For $i \in I$, put
$U_i := \{ 1_f \mid f \in [e_i] \}$.
Then $(U_i)_{i \in I}$ is a partition
of $U$. Moreover, since there are no 
morphisms between different connected
components, we have
$R = \oplus_{i \in I} U_i R U_i$ and
$R_\Delta = \bigoplus_{i \in I} U_i R_\Delta U_i$.
For each $i \in I$, we have
$U_i R U_i = R_{\Gamma_{e_i}}$
and 
$U_i R_\Delta U_i = R_{\Delta_{e_i}}$.
Also, $\Gamma_{e_i}$ is connected. 
Thus, the equivalence of (b) and (c)
follows from Proposition \ref{prop:partition},
and the equivalence of (a) with (b), as well as
of (c) with (d), follows from
Theorem \ref{thm:sepidempotent}(a).
\end{proof}

For the rest of the section, we make
the following additional assumptions:
\begin{itemize}
\item $\Gamma$ is connected, and
\item $R$ is strongly graded by $\Gamma$.
\end{itemize}

It is not clear to the authors whether our results hold when $R$ is 
not strongly graded.
We now generalize some notation
and results from \cite{CLP}.

\begin{definition}\label{def:nuv}
Take $\sigma \in \Gamma$.
Since $R$ is strongly graded, we have
$1_{r(\sigma)} \in R_\sigma 
R_{\sigma^{-1}}$.
Hence there exist $n_\sigma \in \N$,
$u_\sigma^{(i)} \in R_\sigma$, and
$v_{\sigma^{-1}}^{(i)} \in R_{\sigma^{-1}}$,
for $i \in \{ 1,\ldots,n_\sigma \}$,
such that
$1_{r(\sigma)} =
\sum_{i=1}^{n_\sigma} u_\sigma^{(i)}
v_{\sigma^{-1}}^{(i)}$.
Unless otherwise stated, the elements
$u_\sigma^{(i)}$ and
$v_{\sigma^{-1}}^{(i)}$ are fixed.
We also assume that if $e \in \Gamma_0$, 
then $n_e = 1$ and
$u_e^{(1)} = v_e^{(1)} = 1_e$.
Define the additive map
$\gamma_\sigma : R \to R$ by
$\gamma_\sigma(x) =
\sum_{i=1}^{n_\sigma} u_\sigma^{(i)} x
v_{\sigma^{-1}}^{(i)}$, $x \in R$.
\end{definition}

The following two lemmas generalize \cite[Prop. 33 and Prop. 37]{CLP}.
Recall that if $X,Y \subseteq R$, then
$C_X(Y) := 
\{ x \in X \mid xy = yx \ \mbox{for all 
$y \in Y$} \}$.

\begin{lemma}
Take $\sigma \in \Gamma$ and
$x \in C_R(R_0)$. Then
$\gamma_\sigma(x)$ does not depend on the choice
of the elements $u_{\sigma}^{(i)}$ and
$v_{\sigma^{-1}}^{(i)}$.
\end{lemma}

\begin{proof}
Put $e := d(\sigma)$ and $f := r(\sigma)$.
Suppose that $m_\sigma \in \N$,
$s_{\sigma}^{(j)} \in R_{\sigma}$, and
$t_{\sigma^{-1}}^{(j)} \in R_{\sigma^{-1}}$,
for $j \in \{1,\ldots,m_{\sigma} \}$, are
chosen so that
$\sum_{j=1}^{m_{\sigma}} s_{\sigma}^{(j)}
t_{\sigma^{-1}}^{(j)} = 1_f$.
Then
\[
\sum_{i=1}^{n_{\sigma}} u_{\sigma}^{(i)} x
v_{\sigma^{-1}}^{(i)}
=
\sum_{i=1}^{n_{\sigma}} \sum_{j=1}^{m_{\sigma}}
s_{\sigma}^{(j)} t_{\sigma^{-1}}^{(j)}
u_{\sigma}^{(i)} x v_{\sigma^{-1}}^{(i)}.
\]
Since $t_{\sigma^{-1}}^{(j)} u_{\sigma}^{(i)} \in
R_e \subseteq R_0$
and $x \in C_R(R_0)$, the right-hand side equals
\[
\sum_{i=1}^{n_{\sigma}} \sum_{j=1}^{m_{\sigma}}
s_{\sigma}^{(j)} x
t_{\sigma^{-1}}^{(j)} u_{\sigma}^{(i)}
v_{\sigma^{-1}}^{(i)}
=
\sum_{j=1}^{m_{\sigma}} s_{\sigma}^{(j)} x
t_{\sigma^{-1}}^{(j)}
\left( \sum_{i=1}^{n_\sigma} u_\sigma^{(i)}v_{\sigma^{-1}}^{(i)}\right)
=
\sum_{j=1}^{m_{\sigma}} s_{\sigma}^{(j)} x
t_{\sigma^{-1}}^{(j)}. \qedhere
\]
\end{proof}

\begin{lemma}\label{lem: a gamma_tau(x)}
Suppose that $(\sigma,\tau) \in \Gamma_2$ 
and $a \in R_\sigma$.
Take $x \in C_R( R_0 )$. Then
$a \gamma_\tau(x) = 
\gamma_{\sigma\tau}(x) a$.
Furthermore, if $x \in 
C_{R_{\Gamma(d(\sigma))}}(R_0)$,
then $ax = \gamma_\sigma(x) a$.
\end{lemma}

\begin{proof}
Put $f := r(\sigma) = r(\sigma\tau)$. Take
$x \in C_R( R_0 )$. Then
\[
a \gamma_\tau (x)=
\sum_{i=1}^{n_\tau} a u_\tau^{(i)}  x v_{\tau^{-1}}^{(i)}
= \sum_{i=1}^{n_\tau} 1_f a u_\tau^{(i)}
x v_{\tau^{-1}}^{(i)}
=\sum_{i=1}^{n_\tau} \sum_{j=1}^{n_{\sigma\tau}}
u_{\sigma\tau}^{(j)} v_{(\sigma\tau)^{-1}}^{(j)} 
a u_\tau^{(i)} x v_{\tau^{-1}}^{(i)}.
\]
Since $v_{(\sigma\tau)^{-1}}^{(j)}au_\tau^{(i)}
\in R_{d(\tau)} \subseteq R_0$, the last sum equals
\[
\sum_{i=1}^{n_\tau}\sum_{j=1}^{n_{\sigma\tau}}
u_{\sigma\tau}^{(j)} x v_{(\sigma\tau)^{-1}}^{(j)} 
a u_\tau^{(i)}v_{\tau^{-1}}^{(i)} =
\sum_{j=1}^{n_{\sigma\tau}}u_{\sigma\tau}^{(j)}
x v_{(\sigma\tau)^{-1}}^{(j)} a 1_{r(\tau)}
= \gamma_{\sigma\tau}(x)a.
\]
Now put $e := d(\sigma)$ and take
$x \in C_{R_{\Gamma(e)}}(R_0) \subseteq C_R(R_0)$.
By Definition \ref{def:nuv}, we have $\gamma_e(x)=x$.
By the first part with $\tau=e$, 
$ax = a\gamma_e(x) =
\gamma_{\sigma e}(x)a =
\gamma_\sigma(x)a$.
\end{proof}

\begin{lemma}\label{lem:t=gamma_sigma(s),s=gamma_sigma-1(t)}
Take $\sigma \in \Gamma$,
$s \in R_{\Gamma(d(\sigma))}$, and
$t \in R_{\Gamma(r(\sigma))}$.
Suppose that $as = ta$
for each $a \in R_\sigma$.
Then $t = \gamma_\sigma(s)$ and 
$s = \gamma_{\sigma^{-1}}(t)$.
\end{lemma}

\begin{proof}
This follows from
$t=t1_{r(\sigma)}=\sum_{i=1}^{n_\sigma}tu_\sigma^{(i)} v_{\sigma^{-1}}^{(i)}=\sum_{i=1}^{n_\sigma}u_\sigma^{(i)} sv_{\sigma^{-1}}^{(i)}=\gamma_\sigma(s)$
and $s=1_{d(\sigma)}s=\sum_{i=1}^{n_{\sigma^{-1}}}u_{\sigma^{-1}}^{(i)} v_\sigma^{(i)} s=\sum_{i=1}^{n_{\sigma^{-1}}}u_{\sigma^{-1}}^{(i)}t v_\sigma^{(i)}=\gamma_{\sigma^{-1}}(t)$.
\end{proof}

\begin{lemma}
\label{lem: gamma_sigma gamma_tau = gamma_sigmatau}
Suppose that $\sigma,\tau \in \Gamma$ and 
$a\in C_R( R_0 )$.
Then $\gamma_\sigma (\gamma_\tau (a))=
\gamma_{\sigma\tau}(a)$, if $(\sigma,\tau) \in \Gamma_2$,
and $\gamma_\sigma (\gamma_\tau (a)) = 0$,
otherwise.
\end{lemma}

\begin{proof}
The proof of \cite[Prop. 34]{CLP} works in 
our more general situation.
For the convenience of the reader, 
we repeat it here.
Take $\sigma,\tau \in \Gamma$ 
and $a \in C_R(R_0)$. Then
\[
\gamma_\sigma ( \gamma_\tau ( a ) ) 
= \sum_{i=1}^{n_\tau} \gamma_\sigma ( u_\tau^{(i)} a v_{\tau^{-1}}^{(i)} ) 
= \sum_{i=1}^{n_\tau} \sum_{j=1}^{n_\sigma}
u_\sigma^{(j)} u_\tau^{(i)} a v_{\tau^{-1}}^{(i)} v_{\sigma^{-1}}^{(j)}.
\]
If $(\sigma,\tau) \notin \Gamma_2$, then
$u_\sigma^{(j)} u_\tau^{(i)} = 0$ so that
$\gamma_\sigma ( \gamma_\tau (a) ) = 0$.
Now suppose $(\sigma,\tau) \in \Gamma_2$.
Since $u_\sigma^{(j)} u_\tau^{(i)} \in 
R_{\sigma \tau}$, the last sum equals
\begin{displaymath}
\sum_{i=1}^{n_\tau} \sum_{j=1}^{n_\sigma} 
1_{r(\sigma\tau)} u_\sigma^{(j)} u_\tau^{(i)} a v_{\tau^{-1}}^{(i)} v_{\sigma^{-1}}^{(j)} =
\sum_{i=1}^{n_\tau} \sum_{j=1}^{n_\sigma} \sum_{k=1}^{n_{\sigma \tau}}  
u_{\sigma \tau}^{(k)} v_{\tau^{-1} \sigma^{-1}}^{(k)}
u_\sigma^{(j)} u_\tau^{(i)} a v_{\tau^{-1}}^{(i)} v_{\sigma^{-1}}^{(j)}.
\end{displaymath}
Since $a \in C_R(R_0)$ and 
$v_{\tau^{-1} \sigma^{-1}}^{(k)} u_\sigma^{(j)} u_\tau^{(i)} \in R_0$, the last sum equals
\begin{displaymath}
\sum_{i=1}^{n_\tau} \sum_{j=1}^{n_\sigma} \sum_{k=1}^{n_{\sigma\tau}}  
u_{\sigma\tau}^{(k)} a v_{\tau^{-1} \sigma^{-1}}^{(k)}
u_\sigma^{(j)} u_\tau^{(i)} v_{\tau^{-1}}^{(i)} v_{\sigma^{-1}}^{(j)} = 
\sum_{j=1}^{n_\sigma} \sum_{k=1}^{n_{\sigma\tau}}  
u_{\sigma\tau}^{(k)} a v_{\tau^{-1} \sigma^{-1}}^{(k)}
u_\sigma^{(j)} 1_{r(\tau)} v_{\sigma^{-1}}^{(j)}.
\end{displaymath}
Since $d(\sigma)=r(\tau)$, the last sum equals
\begin{displaymath}
\sum_{j=1}^{n_\sigma} \sum_{k=1}^{n_{\sigma\tau}}  
u_{\sigma\tau}^{(k)} a v_{\tau^{-1} \sigma^{-1}}^{(k)}
u_\sigma^{(j)} v_{\sigma^{-1}}^{(j)} = 
\sum_{k=1}^{n_{\sigma\tau}}  
u_{\sigma\tau}^{(k)} a 
v_{(\sigma\tau)^{-1}}^{(k)} 1_{r(\sigma)} 
= \gamma_{\sigma\tau}(a). \qedhere
\end{displaymath}
\end{proof}

\begin{definition}
Take $\sigma,\tau \in \Gamma$.
We say that 
$\sigma$ and $\tau$ are \emph{right $\Lambda$-equivalent} 
if $\sigma \Lambda = \tau \Lambda$.
Since we assume that $\Delta$ is a 
wide subgroupoid of $\Gamma$,
the property of being right 
$\Lambda$-equivalent is 
an equivalence relation on $\Gamma$.
We let $T$ denote a set of
representatives for the
equivalence classes of this relation.
We always assume that we have 
chosen $T$ so that 
$\Gamma_0 \subseteq T$. Take 
$e,f \in \Gamma_0$. We put $T_{e,f} := \{ 
\sigma \in T \mid
\mbox{$d(\sigma)=e$ and $r(\sigma)=f$} \}$
and $T_e := T_{e,e}$.
\end{definition}

\begin{lemma}\label{lem: gamma_sigma=gamma_tau}
Take $x \in C_R(R_\Lambda)$ and $\sigma,\tau \in \Gamma$. 
Suppose that $\sigma$ and $\tau$ are right 
$\Lambda$-equivalent.
Then $\gamma_\sigma(x) = \gamma_\tau(x)$.
\end{lemma}

\begin{proof}
Put $f := r(\sigma)=r(\tau)$. We have
\[ \gamma_\sigma(x)=
\sum_{i=1}^{n_\sigma} u_\sigma^{(i)} x 
v_{\sigma^{-1}}^{(i)} 
= \sum_{i=1}^{n_\sigma} 1_f u_\sigma^{(i)} x 
v_{\sigma^{-1}}^{(i)} =\sum_{i=1}^{n_\sigma} 
\sum_{j=1}^{n_\tau} u_\tau^{(j)}
v_{\tau^{-1}}^{(j)} u_\sigma^{(i)} x 
v_{\sigma^{-1}}^{(i)}.
\]
Since $\sigma$ and $\tau$ are right $\Lambda$-equivalent,
it follows that $v_{\tau^{-1}}^{(j)} u_\sigma^{(i)} \in 
R_{\tau^{-1}\sigma} \subseteq R_\Lambda$. 
Since $x \in C_R(R_\Lambda)$, the last
sum therefore equals
\[ \sum_{i=1}^{n_\sigma}\sum_{j=1}^{n_\tau}
u_\tau^{(j)} x v_{\tau^{-1}}^{(j)} u_\sigma^{(i)} 
v_{\sigma^{-1}}^{(i)} = \sum_{j=1}^{n_\tau}u_\tau^{(j)} x 
v_{\tau^{-1}}^{(j)} 1_f = \gamma_\tau(x).
\qedhere\]
\end{proof}

\begin{definition}\label{def:sigmaaction}
Suppose that $e,e',f \in \Gamma_0$, $\tau \in T_{f,e'}$
and $\sigma \in \Gamma(e',e)$. Let $\tau^\sigma$ 
denote the 
unique element in $T_{f,e}$ that is right
$\Lambda$-equivalent to $\sigma \tau$.
\end{definition}

\begin{lemma}
\label{lem:|T_f,e'|=|T_f,e|}
Suppose that $e,e',e'',f \in \Gamma_0$,
$\sigma \in \Gamma(e',e)$,
and $\sigma' \in \Gamma(e'',e')$.
Then for every $\tau \in T_{f,e''}$,
$(\tau^{\sigma'})^\sigma = \tau^{\sigma \sigma'}$.
In particular, the map
$T_{f,e'} \ni \tau \mapsto \tau^\sigma \in T_{f,e}$
is bijective, with inverse
$T_{f,e} \ni \tau \mapsto \tau^{\sigma^{-1}} \in T_{f,e'}$,
and hence
$|T_{f,e'}| = |T_{f,e}|$.
\end{lemma}


\begin{proof}
Take $\tau \in T_{f,e''}$. Then 
$\tau^{\sigma \sigma'} \Lambda = \sigma \sigma' \tau \Lambda
= \sigma \tau^{\sigma'} \Lambda = 
(\tau^{\sigma'})^\sigma \Lambda$.
Therefore $\tau^{\sigma \sigma'} = (\tau^{\sigma'})^\sigma$
which proves the first part. 
The second part follows
immediately.
\end{proof}

By 
Lemma~\ref{lem:|T_f,e'|=|T_f,e|}, 
for all $e\in\Gamma_0$, $f \in \Gamma_0^{\rm fin} :=  
\{ f \in \Gamma_0 \mid
[\Gamma(f):\Delta(f)] < \infty \}$
we have 
$|T_{f,e}|=|T_f|=[\Gamma(f):\Delta(f)]
< \infty$. We can therefore make 
the following definition.

\begin{definition}
Take $e \in \Gamma_0$. 
We define 
${\rm tr}_{\Gamma/\Delta}^e : R \to R$, by
\[
{\rm tr}_{\Gamma/\Delta}^e(r) = 
\sum_{f \in \Gamma_0^{\rm fin}} 
\sum_{\tau \in T_{f,e}} \gamma_\tau(r),
\quad
r \in R.
\]
\end{definition}

Note that, by Lemma~\ref{lem: gamma_sigma=gamma_tau},
the restriction of
${\rm tr}_{\Gamma/\Delta}^e$
to $C_R(R_\Lambda)$ is independent of the choice
of the sets $T_{f,e}$.

\begin{lemma}
\label{lem: trace and connected components}
$\gamma_\sigma({\rm tr}_{\Gamma/\Delta}^e(r))={\rm tr}_{\Gamma/\Delta}^{e'}(r)$ for
$e,e'\in\Gamma_0$, $\sigma\in\Gamma(e,e')$, and $r\in C_R(R_\Lambda)$.
\end{lemma}

\begin{proof}
By Lemma~\ref{lem: gamma_sigma gamma_tau = gamma_sigmatau}, 
Lemma~\ref{lem: gamma_sigma=gamma_tau}, and
Lemma~\ref{lem:|T_f,e'|=|T_f,e|}, we have
    \begin{align*}
        \gamma_\sigma({\rm tr}_{\Gamma/\Delta}^e(r)) &= 
        \sum_{f \in \Gamma_0^{\rm fin}} \sum_{\tau \in T_{f,e}} 
        \gamma_\sigma(\gamma_\tau(r)) =
        \sum_{f \in \Gamma_0^{\rm fin}} \sum_{\tau \in T_{f,e}} 
        \gamma_{\sigma\tau}(r) \\
        &= 
        \sum_{f \in \Gamma_0^{\rm fin}} \sum_{\tau \in T_{f,e}} 
        \gamma_{\tau^\sigma}(r) = 
        \sum_{f \in \Gamma_0^{\rm fin}} \sum_{\tau' \in T_{f,e'}} 
        \gamma_{\tau'}(r) = 
        {\rm tr}_{\Gamma/\Delta}^{e'}(r).\qedhere
    \end{align*}
\end{proof}

\begin{definition}
For each $\sigma \in \Gamma$, set 
$w_\sigma:=
\sum_{i=1}^{n_\sigma} u_\sigma^{(i)}
\otimes v_{\sigma^{-1}}^{(i)}
\in R \otimes_{R_\Lambda}R$.
\end{definition}

\begin{lemma}
\label{lem: a w_tau = w_sigmatau a}
Take $(\sigma,\tau) \in \Gamma_2$ and $a \in R_\sigma$. 
Then $a w_\tau = w_{\sigma \tau} a$.
\end{lemma}

\begin{proof}
We can use the same proof as in 
\cite[Prop. 41]{CLP}. For the 
convenience of 
the reader, we include it here. We have
\[
a w_\tau = \sum_{i=1}^{n_\tau} a u_\tau^{(i)} \otimes v_{\tau^{-1}}^{(i)} 
        = \sum_{i=1}^{n_\tau} 1_{r(\sigma\tau)} a u_\tau^{(i)} \otimes v_{\tau^{-1}}^{(i)} 
        = \sum_{i=1}^{n_\tau} \sum_{j=1}^{n_{\sigma\tau}} u_{\sigma\tau}^{(j)} v_{(\sigma\tau)^{-1}}^{(j)} a u_\tau^{(i)} \otimes v_{\tau^{-1}}^{(i)}
\]
Since 
$v_{(\sigma\tau)^{-1}}^{(j)} a u_\tau^{(i)}
\in R_0 \subseteq R_\Lambda$, 
the last sum equals 
\begin{align*}
 & \sum_{i=1}^{n_\tau} \sum_{j=1}^{n_{\sigma\tau}} u_{\sigma\tau}^{(j)} \otimes v_{(\sigma\tau)^{-1}}^{(j)} a u_\tau^{(i)} v_{\tau^{-1}}^{(i)} 
        = \sum_{j=1}^{n_{\sigma\tau}} u_{\sigma\tau}^{(j)} \otimes v_{(\sigma\tau)^{-1}}^{(j)} a 1_{r(\tau)} 
        = w_{\sigma\tau} a. \qedhere  
\end{align*}  
\end{proof}

\begin{lemma}
\label{lem: w_sigma=w_tau}
Take $\sigma,\tau \in \Gamma$ with 
$\sigma$ and $\tau$ right 
$\Lambda$-equivalent. Then
$w_\sigma = w_\tau$.
\end{lemma}

\begin{proof}
Put $f := r(\sigma)=r(\tau)$. We have
    \[w_\sigma = \sum_{i=1}^{n_\sigma}u_\sigma^{(i)} \otimes v_{\sigma^{-1}}^{(i)} =
    \sum_{i=1}^{n_\sigma} 1_f u_\sigma^{(i)} \otimes v_{\sigma^{-1}}^{(i)} =
\sum_{i=1}^{n_\sigma}\sum_{j=1}^{n_\tau}u_\tau^{(j)}v_{\tau^{-1}}^{(j)}u_\sigma^{(i)}\otimes v_{\sigma^{-1}}^{(i)}.\]
$\Lambda$-equivalence of $\sigma$ and 
$\tau$ yields
$v_{\tau^{-1}}^{(j)} u_\sigma^{(i)} \in 
R_{\tau^{-1}\sigma} \subseteq R_\Lambda$.
Thus, the last sum equals
\begin{align*}
& \sum_{i=1}^{n_\sigma}\sum_{j=1}^{n_\tau}
u_\tau^{(j)}\otimes 
v_{\tau^{-1}}^{(j)}u_\sigma^{(i)}
v_{\sigma^{-1}}^{(i)}=
\sum_{j=1}^{n_\tau}u_\tau^{(j)}\otimes 
v_{\tau^{-1}}^{(j)} 1_f = w_\tau. \qedhere 
\end{align*}
\end{proof}

Now we prove a result from which 
Theorem \ref{thm:main} follows
immediately.

\begin{theorem}
\label{prop: separability R over R_Delta}
    The following assertions are equivalent:
    \begin{enumerate}[\rm (a)]

\item The functor 
${\rm Res}_{R/{R_\Delta}} :
{}_R {\rm UMod} \to 
{}_{R_\Delta} {\rm UMod}$
is separable.
        
\item The ring extension $R/R_\Delta$ is separable.
        
\item For each $e \in \Gamma_0$, there
is a finite 
subset $F_e$ of
$\Gamma_0^{\rm fin}$
and $r^{e,f} \in C_{R_f}
\left (R_{\Delta(f)} \right)$,
for $f \in F_e$, with 
${\rm tr}_{\Gamma/\Delta}^e(r^e) = 1_{R_e}$,
where $r^e := \sum_{f \in F_e} r^{e,f}$.
        
\item There exist $f \in \Gamma_0$ and 
$r \in  
C_{R_0}\left( R_\Lambda \right)$ with 
${\rm tr}_{\Gamma/\Delta}^{f}( r ) = 1_f$.

\item The ring extension $R/R_\Lambda$ is separable.

\item The functor 
${\rm Res}_{R/{R_\Lambda}} :
{}_R {\rm UMod} \to 
{}_{R_\Lambda} {\rm UMod}$
is separable.

\end{enumerate}
\end{theorem}

\begin{proof}
(a)$\Rightarrow$(b): This follows from
Theorem \ref{thm:sepidempotent}(a).

(b)$\Rightarrow$(c):
Suppose $R/R_\Delta$ is separable. 
By Theorem~\ref{thm:sepidempotent}(a), 
there is, for any $e \in \Gamma_0$,
$y_e \in \sum_{f \in \Gamma_0} 
1_e R 1_f \otimes 1_f R 1_e \subseteq
R \otimes_{R_\Delta} R$ such that 
$\mu(y_e) = 1_e$ and $y_e r = r y_u$,
for $u \in \Gamma_0$ and $r \in 1_e R 1_u$.

Set $M := R \otimes_{R_\Delta} R$. Then
$M = \oplus_{\tau \in T} 
R_{\tau \Lambda} \otimes_{R_{\Delta}}
R_{\Lambda \tau^{-1}}$.
For all $e,f \in \Gamma_0$ and all
$\tau_1,\tau_2 \in T_{f,e}$ we put
$M_{e,f,\tau_1,\tau_2} :=
R_{\tau_1 \Delta(f)} \otimes_{R_\Delta} 
R_{\Delta(f) \tau_2^{-1}}$. Then
\begin{equation}\label{eq:Moplus}
M = \oplus_{e,f \in \Gamma_0,
\tau_1,\tau_2 \in T_{f,e}}
M_{e,f,\tau_1,\tau_2}
\end{equation}
as additive groups. Take 
$\gamma \in \Gamma$. Put 
$M_\gamma := \sum_{(\alpha,\beta) \in \Gamma_2, \alpha \beta = \gamma}
R_\alpha \otimes_{R_\Delta} R_\beta$.
Then
\begin{equation}\label{eq:Moplusagain}
M = \oplus_{\gamma \in \Gamma} M_\gamma
\end{equation}
as additive groups.
For any $e,f \in \Gamma_0$,
$\tau_1,\tau_2 \in T_{f,e}$ and
$\gamma \in \Gamma$ we set
\[
M_{e,f,\tau_1,\tau_2,\gamma} :=
M_{e,f,\tau_1,\tau_2} \cap M_\gamma.
\]
By \eqref{eq:Moplus} and 
\eqref{eq:Moplusagain}, we get
\begin{equation}
M = \oplus_{e,f \in \Gamma_0,
\tau_1,\tau_2 \in T_{f,e},
\gamma \in \tau_1 \Delta(f) \tau_2^{-1}}
M_{e,f,\tau_1,\tau_2,\gamma}
\end{equation}
as additive groups. From this it 
follows that each $y_e$, 
for $e \in \Gamma_0$, can be written
\[
y_e = \sum _{f \in \Gamma_0, 
\tau_1,\tau_2 \in T_{f,e},
\gamma \in \tau_1 \Delta(f) \tau_2^{-1}}
z_{e,f,\tau_1,\tau_2,\gamma}
\]
for some unique 
$z_{e,f,\tau_1,\tau_2,\gamma} 
\in M_{e,f,\tau_1,\tau_2,\gamma}$
such that 
$z_{e,f,\tau_1,\tau_2,\gamma}  = 0$
for all but finitely many of the
indexes $f,\tau_1,\tau_2,\gamma$.

Take $e,g \in \Gamma_0$,
$\gamma \in \Gamma(e)$,
$\sigma \in \Gamma(e,g)$ and
$r_\sigma \in R_\sigma$. 
By the assumptions, $r y_e = y_g r$.
Now we project both $r y_e$ and 
$y_g r$ into 
$M_{g,f,\tau_1^\sigma,\tau_2^\sigma,\sigma\gamma}$ yielding 
\begin{equation}\label{eq:repeatedly}
r_\sigma z_{e,f,\tau_1,\tau_2,\gamma} = 
z_{g,f,\tau_1^\sigma,\tau_2^\sigma,
\sigma \gamma \sigma^{-1}} r_\sigma.
\end{equation}
Now put 
\begin{equation}\label{eq:defxe}
x_e := \sum _{f \in \Gamma_0, 
\tau_1,\tau_2 \in T_{f,e}}
z_{e,f,\tau_1,\tau_2,e}.
\end{equation}
Note that 
\begin{equation}\label{eq:M=0}
M_{e,f,\tau_1,\tau_2,e} = \{ 0 \}
\quad \mbox{for} \quad 
\tau_1 \neq \tau_2.
\end{equation}
Indeed, if 
$M_{e,f,\tau_1,\tau_2,e} \neq \{ 0 \}$,
then $e \in \tau_1 \Delta(f) \tau_2^{-1}$
so that $\tau_2 \in \tau_1 \Delta(f)$
and hence $\tau_1$ and $\tau_2$
are right $\Lambda$-equivalent
which implies that $\tau_1 = \tau_2$.
Thus, by \eqref{eq:defxe} 
and \eqref{eq:M=0},
\begin{equation}\label{eq:xe=ze}
x_e = \sum _{f \in \Gamma_0, 
\tau \in T_{f,e}}
z_{e,f,\tau,\tau,e}.
\end{equation}
Now put
$c_{e,f,\tau,\tau,e} = 
\mu(z_{e,f,\tau,\tau,e})$.
Since $\mu(y_e)=1_e$ and 
$\mu( M_{e,f,\tau_1,\tau_2,\gamma} )
\subseteq R_\gamma$, for 
$\gamma \in \Gamma$,
\eqref{eq:xe=ze} implies that 
$\mu(x_e)=1_e$. Therefore,
\begin{equation}\label{eq:equals1e}
\sum_{f \in \Gamma_0, \tau \in T_{f,e}}
c_{e,f,\tau,\tau,e} = 1_e.
\end{equation}
If we put $\gamma=e$ in
\eqref{eq:repeatedly}, then
$r_\sigma z_{e,f,\tau,\tau,e} =
z_{g,f,\tau^\sigma,\tau^\sigma,g} r_\sigma.$
Applying $\mu$ to this equality yields
\begin{equation}\label{eq:cgf=cef}
c_{g,f,\tau^\sigma,\tau^\sigma,g} r_\sigma 
=
r_\sigma c_{e,f,\tau,\tau,e}.
\end{equation}
Using Lemma
\ref{lem:t=gamma_sigma(s),s=gamma_sigma-1(t)} this implies that
\begin{equation}\label{eq:actionsigma}
c_{g,f,\tau^\sigma,\tau^\sigma,g} =
\gamma_\sigma( c_{e,f,\tau,\tau,e} ).
\end{equation}
Put $c_e := c_{e,e,e,e,e}$.
Specializing \eqref{eq:cgf=cef}
with $e=f=g=\tau$ and $\sigma \in \Delta(e)$ yields
$c_e \in Z(R_{\Delta(e)})$.
Since $c_e \in R_e$,
it therefore follows that 
$c_e \in C_{R_e}(R_{\Delta(e)})$.

Fix $e \in \Gamma_0$.
The set 
$F_e := \{ f \in \Gamma_0 \mid 
\exists \tau \in T_{f,e} \  
c_{e,f,\tau,\tau,e} \neq 0 \}$
is finite and nonempty by 
\eqref{eq:equals1e}.
Take $n_e \in \N$ and
$f_i^{(e)} \in [e]$, for 
$i = 1,\ldots,n_e$, such that
$F_e = \{ f_i^{(e)} \}_{i=1}^{n_e}$.
Fix $i \in \{ 1,\ldots,n_e \}$.
By \eqref{eq:actionsigma}, 
$\gamma_\tau( c_{f_i^{(e)}} ) = 
c_{e,f_i^{(e)},\tau,\tau,e} \neq 0$
for some $\tau \in T_{f_i^{(e)}, e}$.
But then 
$c_{f_i^{(e)}} \neq 0$ also,
which implies that 
$c_{e,f_i^{(e)},\tau,\tau,e} = 
\gamma_\tau( c_{f_i^{(e)}} ) \neq 0$
for all $\tau \in T_{f_i^{(e)}, e}$.
Therefore,
$[\Gamma(f_i^{(e)}) : \Delta(f_i^{(e)})] 
= |T_{f_i^{(e)},e}| < \infty$.

Now set $r^{e,f_i^{(e)}} := 
c_{f_i^{(e)}}$ and
$r^e := \sum_{i=1}^{n_e} r^{e,f_i^{(e)}}$.
Then $r^{e,f_i^{(e)}} \in 
C_{R_{f_i^{(e)}}}
\left( R_{\Delta(f_i^{(e)})} \right)$.
By \eqref{eq:equals1e} and
\eqref{eq:actionsigma}, we get
\[
\begin{array}{rcl}
{\rm tr}_{\Gamma / \Delta}^e(r^e) 
&=&  
\displaystyle
\sum_{i=1}^{n_e} 
\sum_{f \in \Gamma_0^{\rm fin}}
\sum_{\tau \in T_{f,e}} 
\gamma_\tau \left( r^{e,f_i^{(e)}} \right)
=
\displaystyle
\sum_{i=1}^{n_e} 
\sum_{\tau \in T_{f_i^{(e)}, e}}
\gamma_\tau \left( c_{f_i^{(e)}} \right)
\\
&=&
\displaystyle
\sum_{i=1}^{n_e} 
\sum_{\tau \in T_{f_i^{(e)}, e}}
c_{e,f_i^{(e)},\tau,\tau,e}
=
\displaystyle
\sum_{f \in \Gamma_0, \tau \in T_{f,e}} 
c_{e,f,\tau,\tau,e}
= 1_e.
\end{array}
\]

(c)$\Rightarrow$(d): This follows from
the equality $C_{R_0}(R_\Lambda) = 
\oplus_{f \in \Gamma_0} 
C_{R_f}(R_{\Delta(f)})$.

(d)$\Rightarrow$(e):
Suppose that (d) holds.
Take $f \in \Gamma_0$ and 
$r \in C_{R_0}(R_\Lambda)$ such that 
${\rm tr}_{\Gamma/\Delta}^f(r) = 1_f$.
Take $e \in \Gamma_0$. Now put 
\[
x_e := \sum_{f' \in \Gamma_0^{\rm fin}} 
\sum_{\tau \in T_{f',e}} \gamma_\tau(r) w_\tau \in R \otimes_{R_\Lambda} R.
\]
Fix $\sigma \in \Gamma(f, e)$.
By Lemma~\ref{lem: trace and connected components}, we get
\begin{align*}
    \mu( x_e ) &= 
    \sum_{f' \in \Gamma_0^{\rm fin}} 
    \sum_{\tau \in T_{f',e}} 
    \gamma_\tau(r) \mu( w_\tau ) =
    \sum_{f' \in \Gamma_0^{\rm fin}} 
    \sum_{\tau \in T_{f',e}} 
    \gamma_\tau(r) 1_e \\
    &=
    {\rm tr}_{\Gamma/\Delta}^e(r) = \gamma_\sigma({\rm tr}_{\Gamma/\Delta}^f(r)) = \gamma_\sigma(1_f) = 1_e.
\end{align*}
Take $u,v \in \Gamma_0$ and 
$s \in 1_u R 1_v$.
We wish to show that $s x_v = x_u s$. 
It is enough to show $s x_v = x_u s$ 
for $s \in R_\sigma$ and 
$\sigma \in \Gamma(v,u)$.
By Lemmas~\ref{lem: a gamma_tau(x)},
\ref{lem: gamma_sigma=gamma_tau},
\ref{lem:|T_f,e'|=|T_f,e|}, 
\ref{lem: a w_tau = w_sigmatau a},
\ref{lem: w_sigma=w_tau}, we get
\[
\begin{array}{rcl}
s x_v &=& \displaystyle
\sum_{f' \in \Gamma_0^{\rm fin}} 
\sum_{\tau \in T_{f',v}} s \gamma_\tau(r)
w_\tau = 
\sum_{f' \in \Gamma_0^{\rm fin}} 
\sum_{\tau \in T_{f',v}} 
\gamma_{\sigma \tau}(r) 
w_{\sigma \tau} s
\\
&=& \displaystyle 
\sum_{f' \in \Gamma_0^{\rm fin}} 
\sum_{\tau \in T_{f',v}}
\gamma_{\tau^\sigma}(r) 
w_{\tau^\sigma} s =
\sum_{f' \in \Gamma_0^{\rm fin}} 
\sum_{\tau \in T_{f',u}} 
\gamma_\tau(r) w_\tau s = x_u s.
\end{array}
\]
By Theorem~\ref{thm:sepidempotent}(a) $R/R_\Lambda$
is separable.

(e)$\Rightarrow$(f):
Follows from Theorem
\ref{thm:sepidempotent}(a).

(f)$\Rightarrow$(a):
Follows from 
Theorem
\ref{thm:sepidempotent}(a), 
Corollary~\ref{cor: C-B-A}(b) and
$R_\Lambda\subseteq R_\Delta \subseteq R$. 
\end{proof}

\begin{remark}\label{rem:gap}
Suppose that $R/R_\Delta$ is separable.
By Theorem~\ref{prop: separability R over R_Delta}(c)
there is $f \in \Gamma_0$ with 
$[\Gamma(f) : \Delta(f)] < \infty$.
In particular, if $\Gamma = G$ a group
and $\Delta = H$ a subgroup of $G$,
then we can conclude that 
$[G:H] < \infty$. By Theorem~\ref{prop: separability R over R_Delta}(d)
there exists $r \in C_{R_e}(R_H)$ such
that ${\rm tr}_{G/H}(r) = 1_R$.
This is a sharpening of the 
condition from Theorem \ref{thm:miyashita}.
Indeed, if we specialize to the case where 
$R$ is commutative, then 
$C_{R_e}(R_H) = R_e$ whereas
$Z(R_H) = R_H$. Also, our proof of
Theorem 
\ref{prop: separability R over R_Delta}
fills a gap in the proof 
of Theorem \ref{thm:miyashita} from
\cite[Thm. 2.1]{TheohariApostolidi}.
Indeed, in \cite[p.~98]{TheohariApostolidi}
the argument there appears to rely on the 
incorrect conclusion that
$\sum_i s_{kH}^{(i)} \tau_{H t^{-1}}^{(i)}
= 0$ for $s \neq t$. What one can 
conclude is that
\[
\sum_i \left( 
s_{kH}^{(i)} \tau_{H t^{-1}}^{(i)} 
\right)_1 = 0, 
\]
for $s \neq t$, 
where $(\cdot)_1$ denotes projection
onto $R_1$ (using the notation of 
\cite{TheohariApostolidi}).
\end{remark}

We summarize the conclusions of the
above remark in the following corollary.

\begin{corollary}\label{cor:miyashita}
Let $R$ be a unital ring that is 
strongly graded by the group $G$. Suppose that $H$ 
is a subgroup of $G$. 
Then $R/R_H$ is separable
if and only if $[G:H]$ is finite and 
there exists $r \in C_{R_e}(R_H)$ 
such that $\tr_{G/H}(r) = 1_R$.
\end{corollary}

\begin{corollary}
\label{prop: sufficent conditions for separability}
Suppose that there is $e \in \Gamma_0$ 
such that $R_{\Gamma(e)}/R_{\Delta(e)}$ is 
separable. Then $R/R_\Delta$ is separable.
\end{corollary}

\begin{proof}
Put $G := \Gamma(e)$ and 
$H := \Delta(e)$. By 
Corollary \ref{cor:miyashita},
$[G:H] < \infty$ and there is
$r\in C_{R_{e}}(R_H) \subseteq 
C_{R_0}(R_\Lambda)$ with 
${\rm tr}_{G/H}(r) = 1_e$. 
Since $T_e$ is a left transversal of 
$H$ in $G$, we get
${\rm tr}^e_{\Gamma/\Delta}(r) = 
\sum_{f \in \Gamma_0^{\rm fin}} \sum_{\tau\in T_{f,e}}\gamma_{\tau}(r) =
\sum_{\tau\in T_e}\gamma_{\tau}(r) 
= {\rm tr}_{G/H}(r) = 1_e.$
Therefore,
by Theorem~\ref{prop: separability R over R_Delta}(d), $R/{R_\Delta}$ is 
separable.
\end{proof}

\section{Separability: normal case}\label{sec:separabilitynormalcase}

In this section, we keep the 
notation and assumptions 
from the previous section 
and we make the following additional 
assumption:
\begin{itemize}

\item $\Delta$ is a \emph{normal}
subgroupoid of $\Gamma$.

\end{itemize}
Recall that this means that $\Delta$ is a 
wide subgroupoid
of $\Gamma$ such that, for all $e,f\in\Gamma_0$
and $\sigma \in \Gamma(e,f)$, 
we have $\sigma \Delta(e) \sigma^{-1}
\subseteq \Delta(f)$.
In particular, for each $e \in \Gamma_0$,
$\Delta(e)$ is a normal subgroup of $\Gamma(e)$
and we can form the quotient group
$\Gamma(e)/\Delta(e)$; if $\sigma \in \Gamma(e)$, 
we let $\overline{\sigma}$ denote the coset
$\sigma \Delta(e)$. 
Using this notation, we can describe the
sets $T_{f,e}$ from 
Lemma~\ref{lem:|T_f,e'|=|T_f,e|} more concretely:

\begin{lemma}
\label{lem: |T_e,e|=|T_f,e|}
Let $e,f\in\Gamma_0$ and $\tau\in T_{f,e}$. 
The map $\alpha :
\Gamma(e)/\Delta(e) \ni \overline{\sigma}
\mapsto \tau^\sigma \in T_{f,e}$ is a 
bijection with inverse
$\beta: T_{f,e} \ni \rho \mapsto 
\overline{\rho \tau^{-1}}
\in \Gamma(e)/\Delta(e)$.
\end{lemma}

\begin{proof}
First we show that $\alpha$ is well defined.
Take $\sigma_1,\sigma_2 \in \Gamma(e)$
with $\overline{\sigma_1} = \overline{\sigma_2}$.
Then $\sigma_2^{-1} \sigma_1 \in \Delta(e)$.
Since $\Delta$ is a normal subgroupoid of
$\Gamma$, we get
\[
(\sigma_2 \tau)^{-1} \sigma_1 \tau
=
\tau^{-1} \sigma_2^{-1} \sigma_1 \tau
\in
\tau^{-1} \Delta(e) \tau
\subseteq \Delta(f),
\]
so $\sigma_1\tau$ and $\sigma_2\tau$ are right
$\Lambda$-equivalent. Hence
$\tau^{\sigma_1}=\tau^{\sigma_2}$.
Thus $\alpha$ is well defined.
The map $\beta$ is well defined since, for
$\rho \in T_{f,e} \subseteq \Gamma(f,e)$, we have $\rho \tau^{-1} \in \Gamma(e)$.

Now let $\overline{\sigma} \in \Gamma(e)/\Delta(e)$.
Since $\tau^\sigma$ is right $\Lambda$-equivalent to
$\sigma\tau$, 
$\tau^\sigma \tau^{-1} \in 
\sigma \Delta(e)$, and therefore
$(\beta \circ \alpha)(\overline{\sigma})
=
\beta(\tau^\sigma)
=
\overline{\tau^\sigma \tau^{-1}}
=
\overline{\sigma}$.
Hence $\beta \circ \alpha = {\rm id}_{\Gamma(e)/\Delta(e)}$.

Next, let $\rho \in T_{f,e}$.
Then
$(\alpha \circ \beta)(\rho)
=
\alpha(\overline{\rho\tau^{-1}})
=
\tau^{\rho\tau^{-1}}$.
Since
$(\rho\tau^{-1})\tau = \rho$,
the elements $\tau^{\rho\tau^{-1}}$ and $\rho$
belong to $T_{f,e}$ and are right $\Lambda$-equivalent.
By uniqueness of representatives in $T_{f,e}$, it follows that
$\tau^{\rho\tau^{-1}}=\rho$.
Thus
$\alpha \circ \beta = {\rm id}_{T_{f,e}}$.
\end{proof}

\begin{example}
Lemma~\ref{lem: |T_e,e|=|T_f,e|} does
not hold if we omit the hypothesis
that $\Delta$ is a normal subgroupoid of
$\Gamma$. Indeed, let $G$ be a
nontrivial group with
identity element $e_G$.
Consider the groupoid
$\Gamma := \{1,2\}\times G\times \{1,2\}$
with objects $\Gamma_0=\{1,2\}$,
morphisms
$\Gamma(j,i)=\{(i,g,j):g\in G\}$,
and composition defined by
$(i,g,j)(j,h,k) = (i,gh,k)$
for all $g,h\in G$ and $i,j,k\in\{1,2\}$.
Consider the wide subgroupoid
\[
\Delta :=
\{(1,g,1):g\in G\}\cup\{(2,e_G,2)\}
\]
of $\Gamma$.
If we put $e:=1$ and $f:=2$, then
$|T_{e,e}|=1$, but $|T_{f,e}| = |G| > 1$.
Indeed, $\Delta(e)=G$, so there is only one right $\Lambda$-equivalence class in $\Gamma(e,e)$, whereas $\Delta(f)=\{e_G\}$, so distinct elements of $\Gamma(e,f)$ remain in distinct right $\Lambda$-equivalence classes.
\end{example}

\begin{lemma}
\label{lem: gamma_sigma restricts to}
Take $e,f \in \Gamma_0$ and $\sigma \in 
\Gamma(e,f)$. 
The map $\gamma_\sigma : R \to R$ restricts to
an additive isomorphism 
$\gamma_\sigma' :
R_{\Delta(e)} \to R_{\Delta(f)}$. 
This map in turn restricts to a ring 
iso\-morphism
$\gamma_\sigma'' : 
Z( R_{\Delta(e)} ) \to Z( R_{\Delta(f)} )$
which restricts to a ring isomorphism 
$\gamma_\sigma''' : C_{R_e}(R_{\Delta(e)}) 
\to C_{R_f}(R_{\Delta(f)})$.
\end{lemma}

\begin{proof}
Since $\Delta$ is a
normal subgroupoid of $\Gamma$, it follows that
$\sigma \Delta(e) \sigma^{-1} \subseteq \Delta(f)$.
Hence $\gamma_\sigma'( R_{\Delta(e)} ) = \sum_{i=1}^{n_\sigma}u_\sigma^{(i)} R_{\Delta(e)} v_{\sigma^{-1}}^{(i)} \subseteq 
R_\sigma R_{\Delta(e)} R_{\sigma^{-1}}
\subseteq 
R_{\sigma \Delta(e) \sigma^{-1}} \subseteq 
R_{\Delta(f)}$.
From Lemma~\ref{lem: gamma_sigma gamma_tau = gamma_sigmatau}, it follows that $\gamma_\sigma'$
is bijective with inverse $\gamma_{\sigma^{-1}}'$.
Take $x \in Z(R_{\Delta(e)})$ and 
$y \in Z( R_{\Delta(f)} )$. Then
\[
\gamma_\sigma''(x)y
= \sum_{i=1}^{n_\sigma} u_\sigma^{(i)} x v_{\sigma^{-1}}^{(i)} y = \sum_{i=1}^{n_\sigma}u_\sigma^{(i)} 
x v_{\sigma^{-1}}^{(i)} y 1_f = 
\sum_{i,j=1}^{n_\sigma}u_\sigma^{(i)} 
x v_{\sigma^{-1}}^{(i)} y u_\sigma^{(j)} v_{\sigma^{-1}}^{(j)}.
\]
Since $\Delta$ is a
normal subgroupoid of $\Gamma$, it follows that
$\sigma^{-1} \Delta(f) \sigma \subseteq \Delta(e)$.
Therefore,
$v_{\sigma^{-1}}^{(i)} y u_\sigma^{(j)}
\in R_{\sigma^{-1}} R_{\Delta(f)}
R_{\sigma} \subseteq R_{\sigma^{-1} \Delta(f)
\sigma} \subseteq 
R_{\Delta(e)}$, and the last sum equals
\[
\sum_{i,j=1}^{n_\sigma}u_\sigma^{(i)} 
v_{\sigma^{-1}}^{(i)} y u_\sigma^{(j)} x
v_{\sigma^{-1}}^{(j)} = 
\sum_{j=1}^{n_\sigma} y u_\sigma^{(j)} x v_{\sigma^{-1}}^{(j)} = y \gamma_\sigma''(x).\]
Thus $\gamma_\sigma'' ( Z(R_{\Delta(e)} )) 
\subseteq Z( R_{\Delta(f)} )$. 
By Lemma~\ref{lem: gamma_sigma gamma_tau = gamma_sigmatau}, $\gamma_\sigma''$
is bijective with inverse $\gamma_{\sigma^{-1}}''$.
Now we show that $\gamma_\sigma''$ respects
multiplication.
Take $x,x' \in Z(R_{\Delta(e)})$. Then
\[
\gamma_{\sigma}''(x x') = 
\sum_{i=1}^{n_\sigma} u_\sigma^{(i)} x x'
v_{\sigma^{-1}}^{(i)} = 
\sum_{i=1}^{n_\sigma} 1_f u_\sigma^{(i)} x x'
v_{\sigma^{-1}}^{(i)} 
= \sum_{i=1}^{n_\sigma} \sum_{j=1}^{n_\sigma}
u_\sigma^{(j)} v_{\sigma^{-1}}^{(j)} 
u_\sigma^{(i)} x x' v_{\sigma^{-1}}^{(i)}.
\]
Since $x \in Z(R_{\Delta(e)})$ and $v_{\sigma^{-1}}^{(j)} u_\sigma^{(i)} \in 
R_{\Delta(e)}$, the last sum equals
\[
\sum_{i=1}^{n_\sigma} \sum_{j=1}^{n_\sigma}
u_\sigma^{(j)} x v_{\sigma^{-1}}^{(j)} 
u_\sigma^{(i)} x' v_{\sigma^{-1}}^{(i)} 
= \sum_{j=1}^{n_\sigma} u_\sigma^{(j)} x 
v_{\sigma^{-1}}^{(j)} \sum_{i=1}^{n_\sigma} 
u_\sigma^{(i)} x' v_{\sigma^{-1}}^{(i)} 
= \gamma_\sigma''(x) \gamma_\sigma''(x').
\]
Also $\gamma_\sigma''(1_e) = 
\sum_{i=1}^{n_\sigma} u_\sigma^{(i)} 1_e
v_{\sigma^{-1}}^{(i)} = 
\sum_{i=1}^{n_\sigma} u_\sigma^{(i)} 
v_{\sigma^{-1}}^{(i)} = 1_f$.
Finally, if
$x \in C_{R_e}(R_{\Delta(e)})$,
then $x \in Z(R_{\Delta(e)})$ and
$\gamma_\sigma(x)\in R_f$.
Since $\gamma_\sigma''(x)\in Z(R_{\Delta(f)})$, it follows that
$\gamma_\sigma''' :
C_{R_e}(R_{\Delta(e)})
\to
C_{R_f}(R_{\Delta(f)})$
is a ring isomorphism.
\end{proof}

\begin{theorem}
\label{prop: separability R over R_Delta, normal case}
The following assertions are equivalent:
\begin{enumerate}[\rm (a)]

\item The functor 
${\rm Res}_{R/{R_\Delta}} :
{}_R {\rm UMod} \to 
{}_{R_\Delta} {\rm UMod}$
is separable.
        
\item The ring extension $R/R_\Delta$ is separable.
        
\item For each $e\in\Gamma_0$, we have  
$[\Gamma(e):\Delta(e)] < \infty$ and 
there exists $r \in C_{R_e}(R_{\Delta(e)})$ 
such that
${\rm tr}^e_{\Gamma(e)/\Delta(e)}(r)=1_e$.
        
\item For each $e\in\Gamma_0$, the ring extension 
$R_{\Gamma(e)} / R_{\Delta(e)}$ is separable.

\item For each $e \in \Gamma_0$,
${\rm Res}_{R_{\Gamma(e)}/{R_{\Delta(e)}}} :
{}_{R_{\Gamma(e)}} {\rm UMod} \to 
{}_{R_{\Delta(e)}} {\rm UMod}$
is separable.

\end{enumerate}
\end{theorem}

\begin{proof}
The equivalences
(a)$\Leftrightarrow$(b) and 
(d)$\Leftrightarrow$(e) follow from
Theorem
\ref{prop: separability R over R_Delta}.

(b)$\Rightarrow$(c):
Suppose that $R/R_\Delta$ is separable. 
Take $e\in \Gamma_0$.
By Theorem~\ref{prop: separability R over R_Delta},
there exist $n \in \N$,
$f_1,...,f_n \in \Gamma_0$ and 
$s \in C_{R_0}( R_\Lambda )$ such that
$|T_{f_i,e}|<\infty$, for
$i=1,...,n$, and $1_e=\sum_{i=1}^n\sum_{\tau\in T_{f_i,e}}\gamma_\tau(s)$. 
By Lemma~\ref{lem: |T_e,e|=|T_f,e|}, 
$[\Gamma(e):\Delta(e)] = 
| T_{f_i,e} | < \infty$.
For each $i=1,...,n$, fix $\tau_i\in T_{f_i,e}$. 
Put $r := \sum_{i=1}^n\gamma_{\tau_i}(s)$. 
By Lemma~\ref{lem: gamma_sigma restricts to}, 
$r \in C_{R_e}( R_{\Delta(e)} )$. We have
\[
\tr_{\Gamma(e)/\Delta(e)}(r) = 
\sum_{i=1}^n \tr_{\Gamma(e)/\Delta(e)}
(\gamma_{\tau_i}(s)) = 
\sum_{i=1}^n \sum_{\overline{\sigma}\in \Gamma(e)/\Delta(e)} \gamma_\sigma(\gamma_{\tau_i}(s)).
\]
By Lemmas \ref{lem: gamma_sigma gamma_tau = gamma_sigmatau}, 
\ref{lem: gamma_sigma=gamma_tau} and 
\ref{lem: |T_e,e|=|T_f,e|}, 
the last sum equals
\[
\sum_{i=1}^n \sum_{\overline{\sigma} \in 
\Gamma(e)/\Delta(e)} \gamma_{\sigma\tau_i}(s) = 
\sum_{i=1}^n \sum_{\overline{\sigma} \in 
\Gamma(e)/\Delta(e)}
\gamma_{\tau_i^\sigma}(s) = 
\sum_{i=1}^n 
\sum_{\tau \in T_{f_i,e}}\gamma_{\tau}(s) 
= 1_e.
\]

(c)$\Rightarrow$(d):
This follows from 
Corollary \ref{cor:miyashita}.

(d)$\Rightarrow$(b):
This follows from 
Corollary 
\ref{prop: sufficent conditions for separability}.
\end{proof}

\section{Separability: object crossed products}\label{sec:objectcrossedproducts}

In this section, we
apply our previous results
to the context of object crossed products.
For a unital ring $A$,
we let $U(A)$ denote the group of
multiplicative units of $A$.
We first recall some notions from
\cite{CLP}.

By an \emph{object crossed system}
we mean a quadruple $(A,\Gamma,\alpha,\beta)$ where
$\Gamma$ is a groupoid,
$A = \left( A_e \right)_{e\in\Gamma_0}$
is a family of unital rings,
$\alpha = \left(
\alpha_\sigma : A_{d(\sigma)}\to A_{r(\sigma)}
\right)_{\sigma\in\Gamma}$
is a family of unital ring isomorphisms, and
$\beta = \left(
\beta_{\sigma,\tau}
\in
U(A_{r(\sigma)})
\right)_{(\sigma,\tau) \in \Gamma_2}$
is a family of invertible elements,
satisfying the following four conditions:
\begin{itemize}

\item[(O1)]
$\alpha_e=\id_{A_e}$ for all $e\in\Gamma_0$;

\item[(O2)]
$\beta_{\sigma,d(\sigma)}=\beta_{r(\sigma),\sigma}= 1_{A_{r(\sigma)}}$
for all $\sigma\in\Gamma$;

\item[(O3)]
$\alpha_\sigma(\alpha_\tau(a))=
\beta_{\sigma,\tau}\alpha_{\sigma\tau}(a)\beta_{\sigma,\tau}^{-1}$
for all $(\sigma,\tau)\in\Gamma_2$ and $a\in A_{d(\tau)}$;

\item[(O4)]
$\beta_{\sigma,\tau}\beta_{\sigma\tau,\rho}=
\alpha_\sigma(\beta_{\tau,\rho})\beta_{\sigma,\tau\rho}$
for all $(\sigma,\tau),(\tau,\rho)\in\Gamma_2$.

\end{itemize}
Let $(A,\Gamma,\alpha,\beta)$ be an
object crossed system.
Let $\{ u_\sigma \}_{\sigma\in\Gamma}$
be a copy of $\Gamma$.
By the \emph{object crossed product}
$A \rtimes^\alpha_\beta \Gamma$ defined by
$(A,\Gamma,\alpha,\beta)$
we mean the set of
formal sums $\sum_{\sigma \in \Gamma}
a_\sigma u_\sigma$,
where $a_\sigma \in A_{r(\sigma)}$
for each $\sigma \in \Gamma$, and
$a_\sigma = 0$ for all but finitely many
$\sigma \in \Gamma$. If
$\sum_{\sigma \in \Gamma} a_\sigma 
u_\sigma$
and
$\sum_{\sigma \in \Gamma} a_\sigma' 
u_\sigma$
are two such sums, we define
\[
\sum_{\sigma \in \Gamma} a_\sigma u_\sigma
+
\sum_{\sigma \in \Gamma} a_\sigma' u_\sigma
:=
\sum_{\sigma \in \Gamma}
(a_\sigma + a_\sigma') u_\sigma.
\]
The product of two such sums
is defined to be
the additive extension of
\[
\left( a_\sigma u_\sigma \right) \cdot
\left( a_{\tau}' u_{\tau} \right) :=
a_\sigma \alpha_\sigma ( a_{\tau}' )
\beta_{\sigma,\tau} u_{\sigma \tau},
\]
when $(\sigma,\tau) \in \Gamma_2$, and
$a_\sigma u_\sigma \cdot a_{\tau}' 
u_{\tau} := 0$, otherwise.
By \cite[Prop.~16]{CLP},
$A\rtimes^\alpha_\beta\Gamma$ is an object
unital strongly $\Gamma$-graded ring if,
for each $\sigma \in \Gamma$, we set
$(A\rtimes^\alpha_\beta\Gamma)_\sigma :=
A_{r(\sigma)}u_\sigma$.
We distinguish the following subclasses
of object crossed products.

\begin{itemize}

\item
If $\beta$ is \emph{trivial}, that is, if
$\beta_{\sigma,\tau} = 1_{A_{r(\sigma)}}$
for all $(\sigma,\tau) \in \Gamma_2$,
then we say that the corresponding object crossed
product is an \emph{object skew
groupoid ring}, and we
denote it by $A \rtimes^{\alpha} \Gamma$.

\item
If $\alpha$ is \emph{trivial}, that is, if
all the rings $A_e$, for $e \in \Gamma_0$,
coincide with the same ring $B$, and for all
$\sigma \in \Gamma$,
the map $\alpha_\sigma : B \to B$ is
the identity,
then we say that the corresponding
object crossed product is an
\emph{object twisted groupoid ring},
and we denote it by
$B \rtimes_{\beta} \Gamma$.

\item
If \emph{both} $\alpha$ and $\beta$
are trivial, then the corresponding
object crossed product is the \emph{groupoid ring}
of $\Gamma$ over $B$, and we denote
it by $B[\Gamma]$.

\end{itemize}
For the rest of this section,
we make the following assumptions:
\begin{itemize}

\item $(A,\Gamma,\alpha,\beta)$ denotes a
fixed object crossed product system
with $\Gamma$ connected;

\item $R :=
A \rtimes^\alpha_\beta \Gamma$
is the object crossed product
defined by $(A,\Gamma,\alpha,\beta)$;

\item $\Delta$ denotes a wide
subgroupoid of $\Gamma$, and
we set
$\Lambda := {\rm Iso}(\Delta)
= \cup_{e \in \Gamma_0} \Delta(e)$;

\item
$R_\Delta :=
A \rtimes^{\alpha_\Delta}_{\beta_\Delta}
\Delta$
is the object crossed product defined
by the object crossed system
$(A,\Delta,\alpha_\Delta,\beta_\Delta)$,
where
$\alpha_\Delta =
( \alpha_\sigma )_{\sigma \in \Delta}$
and $\beta_\Delta =
( \beta_{\sigma,\tau} )_{(\sigma,\tau) \in \Delta_2}$.

\end{itemize}

We now study separability of
$R/R_\Delta$ using Theorem
\ref{prop: separability R over R_Delta}.
To this end, take
$\sigma \in \Gamma$. By (O1) and (O4)
it follows that
$\alpha_\sigma(\beta_{\sigma^{-1},\sigma})
= \beta_{\sigma,\sigma^{-1}}$.
Thus
$1_{r(\sigma)}u_{r(\sigma)}=(1_{r(\sigma)}u_{\sigma}) \cdot 
(\beta_{\sigma^{-1},\sigma}^{-1}u_{\sigma^{-1}})$.
Hence, using the notation from
Definition \ref{def:nuv}, we can put
$n_\sigma:=1$, 
$u_\sigma^{(1)}:=1_{A_r(\sigma)}u_{\sigma}$ 
and 
$v_{\sigma^{-1}}^{(1)}:=\beta_{\sigma^{-1},\sigma}^{-1}u_{\sigma^{-1}}$. Then,
for any $a \in A_{d(\sigma)}$, we have
\begin{equation}\label{eq:crossedaction}
\begin{array}{rcl}
\gamma_\sigma(au_{d(\sigma)}) &=&
(1_{A_r(\sigma)}u_{\sigma})(au_{d(\sigma)})(\beta_{\sigma^{-1},\sigma}^{-1}u_{\sigma^{-1}}) 
\\
&=&
(\alpha_\sigma(a)u_{\sigma})(\beta_{\sigma^{-1},\sigma}^{-1}u_{\sigma^{-1}})
\\
&=&
\alpha_\sigma(a) 
\alpha_\sigma( \beta_{\sigma^{-1},\sigma} )
u_{r(\sigma)} \\
&=& \alpha_\sigma(a)u_{r(\sigma)}.
\end{array}
\end{equation}
If $e \in \Gamma_0$, and 
$X \subseteq A_e$, we put 
$X^{\Delta(e)} := \{ x \in X \mid  
\alpha_\sigma(x) = x
\text{ for all } \sigma \in \Delta(e) \}$.

\begin{lemma}\label{lem:fixed}
Take $e \in \Gamma_0$. Then
$C_{R_e}( R_{\Delta(e)} ) = 
\left( Z(A_e)^{\Delta(e)} \right) u_e$.
\end{lemma}

\begin{proof}
Put
$I := C_{R_e}(R_{\Delta(e)})$ and
$J := \left( Z(A_e)^{\Delta(e)} \right) 
u_e$.

First we show that $I \subseteq J$.
Take $a u_e \in I$ for some $a \in A_e$.
Take $b \in A_e$. Since
$R_e \subseteq R_{\Delta(e)}$, we have
$I \subseteq Z(R_e)$. Therefore
$(a u_e)(b u_e) = (b u_e)(a u_e)$,
so $(ab)u_e = (ba)u_e$, 
and hence $ab = ba$. Thus $a \in Z(A_e)$.
Now take $\sigma \in \Delta(e)$.
Then $1_{A_e}u_\sigma \in R_{\Delta(e)}$, and therefore we get
$(a u_e)(1_{A_e} u_\sigma) =
(1_{A_e} u_\sigma)(a u_e)$.
This implies that
$a u_\sigma = \alpha_\sigma(a) u_\sigma$,
and hence $a = \alpha_\sigma(a)$.
Thus $a u_e \in J$.

Now we show that $J \subseteq I$.
To this end, take $c u_e \in J$ for some
$c \in Z(A_e)^{\Delta(e)}$.
Take $d u_\sigma \in R_{\Delta(e)}$ for
some $d \in A_e$ and $\sigma \in \Delta(e)$.
Then
$(d u_\sigma)(c u_e) =
d \alpha_\sigma(c) u_\sigma =
d c u_\sigma =
c d u_\sigma =
(c u_e)(d u_\sigma)$.
Hence $c u_e \in I$.
This proves the claim.
\end{proof}

\begin{lemma}\label{lem:welldefinedtr}
Let $e,f \in \Gamma_0$ and suppose that
$\tau_1,\tau_2 \in \Gamma(f,e)$ are right
$\Lambda$-equivalent. 
Take $a \in Z(A_f)^{\Delta(f)}$.
Then $\alpha_{\tau_1}(a) = 
\alpha_{\tau_2}(a)$.
\end{lemma}

\begin{proof}
Since $\tau_1$ and $\tau_2$ are right
$\Lambda$-equivalent, there is
$\sigma \in \Delta(f)$ such that
$\tau_1 = \tau_2 \sigma$.
Since
$\alpha_{\tau_2} : A_f \to A_e$
is a ring isomorphism, it restricts to
a ring isomorphism
$Z(A_f) \to Z(A_e)$.
Moreover, from
$a \in Z(A_f)^{\Delta(f)}$,
we get $\alpha_\sigma(a) = a$.
Thus, by (O3),
$\alpha_{\tau_1}(a)
=
\alpha_{\tau_2 \sigma}(a)
=
\beta_{\tau_2,\sigma}^{-1}
\alpha_{\tau_2}(\alpha_\sigma(a))
\beta_{\tau_2,\sigma}
=
\beta_{\tau_2,\sigma}^{-1}
\alpha_{\tau_2}(a)
\beta_{\tau_2,\sigma}$.
Since $\alpha_{\tau_2}(a) \in Z(A_e)$ and
$\beta_{\tau_2,\sigma} \in U(A_e)$, we get
$\beta_{\tau_2,\sigma}^{-1}
\alpha_{\tau_2}(a)
\beta_{\tau_2,\sigma}
=
\alpha_{\tau_2}(a)$.
Hence
$\alpha_{\tau_1}(a)=\alpha_{\tau_2}(a)$.
\end{proof}

\begin{definition}
Take $e,f \in \Gamma_0$.
Then we define
${\rm tr}_{\Gamma/\Delta}^{f,e} :
Z(A_f)^{\Delta(f)} \to
Z(A_e)^{\Delta(e)}$
by
${\rm tr}_{\Gamma/\Delta}^{f,e}(a) =
\sum_{\tau \in T_{f,e}}
\alpha_\tau(a)$,
for $a \in Z(A_f)^{\Delta(f)}$.
\end{definition}

\begin{proposition}
Take $e,f \in \Gamma_0$.
Then ${\rm tr}_{\Gamma/\Delta}^{f,e}$
is well defined.
\end{proposition}

\begin{proof}
Since each
$\alpha_\tau : A_f \to A_e$,
$\tau \in T_{f,e}$,
is a ring isomorphism, it restricts
to a ring isomorphism
$Z(A_f) \to Z(A_e)$.
Therefore,
${\rm tr}_{\Gamma/\Delta}^{f,e}$ maps
$Z(A_f)$ into $Z(A_e)$.
By Lemma \ref{lem:welldefinedtr},
${\rm tr}_{\Gamma/\Delta}^{f,e}$ is
independent of the choice of the
set $T_{f,e}$.
Take $a \in Z(A_f)^{\Delta(f)}$ and
$\sigma \in \Delta(e)$.
Since $\alpha_{\sigma\tau}(a)\in Z(A_e)$
for $\tau \in T_{f,e}$, axiom (O3) implies that
\[
\alpha_\sigma\bigl(
{\rm tr}_{\Gamma/\Delta}^{f,e}(a)
\bigr)
=
\sum_{\tau \in T_{f,e}}
\alpha_\sigma(\alpha_\tau(a))
=
\sum_{\tau \in T_{f,e}}
\beta_{\sigma,\tau}\alpha_{\sigma\tau}(a)\beta_{\sigma,\tau}^{-1}
=
\sum_{\tau \in T_{f,e}}
\alpha_{\sigma\tau}(a).
\]
Now \(\sigma\tau\) and \(\tau^\sigma\) are right
\(\Lambda\)-equivalent, so by Lemma \ref{lem:welldefinedtr},
$\alpha_{\sigma\tau}(a)=
\alpha_{\tau^\sigma}(a)$.
Hence
$\alpha_\sigma\bigl(
{\rm tr}_{\Gamma/\Delta}^{f,e}(a)
\bigr)
=
\sum_{\tau \in T_{f,e}}
\alpha_{\tau^\sigma}(a)$.
By Lemma \ref{lem:|T_f,e'|=|T_f,e|},
the map \(\tau\mapsto \tau^\sigma\) is a bijection of \(T_{f,e}\), so the last sum equals
$\sum_{\tau \in T_{f,e}}
\alpha_\tau(a)
=
{\rm tr}_{\Gamma/\Delta}^{f,e}(a)$.
Thus
${\rm tr}_{\Gamma/\Delta}^{f,e}(a)
\in Z(A_e)^{\Delta(e)}$.
\end{proof}

\begin{lemma}\label{lem:tracegammadelta}
If $e \in \Gamma_0$ and
$a \in Z(A_e)^{\Delta(e)}$, then
${\rm tr}_{\Gamma(e)/\Delta(e)}(a u_e) =
{\rm tr}_{\Gamma/\Delta}^{e,e}(a)\, u_e$.
\end{lemma}

\begin{proof}
This follows immediately from
\eqref{eq:crossedaction}.
\end{proof}

\begin{theorem}\label{thm:crossed}
Let $R/R_\Delta$ be an extension of
object crossed products.
\begin{enumerate}[{\rm (a)}]

\item $R/R_\Delta$ is
separable if and only if, for each $e \in \Gamma_0$,
there exist a finite subset $F_e$ of
$\Gamma_0^{\rm fin}$
and elements $a_f^e \in Z(A_f)^{\Delta(f)}$,
for $f \in F_e$, with
$\sum_{f \in F_e}
{\rm tr}_{\Gamma/\Delta}^{f,e}(a_f^e) = 1_{A_e}$.

\item Suppose that there are
$f \in \Gamma_0$ and
$a \in Z(A_f)^{\Delta(f)}$ such that
$[\Gamma(f):\Delta(f)] < \infty$ and
${\rm tr}_{\Gamma/\Delta}^{f,f}(a)
= 1_{A_f}$.
Then $R/R_\Delta$ is separable.

\item If $\Delta$ is a normal
subgroupoid of $\Gamma$, then
$R/R_\Delta$ is separable if and only if,
for each $e \in \Gamma_0$,
$[\Gamma(e):\Delta(e)] < \infty$ and
there is $a \in Z(A_e)^{\Delta(e)}$
with ${\rm tr}_{\Gamma/\Delta}^{e,e}(a)
= 1_{A_e}$.

\item If $R = B\rtimes_\beta\Gamma$
is an object twisted groupoid
ring, then $R/R_\Delta$ is separable
if and only if, for each $e \in \Gamma_0$,
there exist a finite subset $F_e$ of
$\Gamma_0^{\rm fin}$
and elements $b_f^e \in Z(B)$,
for $f \in F_e$, such that
$\sum_{f \in F_e}
[\Gamma(f):\Delta(f)]\, b_f^e = 1_B$.

\item If $\Delta$ is a normal
subgroupoid of $\Gamma$ and
$R = B\rtimes_\beta\Gamma$
is an object twisted groupoid ring,
then $R/R_\Delta$ is separable
if and only if, for each $e \in \Gamma_0$,
$[\Gamma(e) : \Delta(e)]$ is finite and
invertible in $B$.

\item If $R = B\rtimes_\beta\Gamma$
is an object twisted groupoid ring and
there is $f \in \Gamma_0$ such that
$[\Gamma(f):\Delta(f)]$ is finite and
invertible in $B$, then
$R/R_\Delta$ is separable.

\end{enumerate}
\end{theorem}

\begin{proof}
(a) follows from
Theorem \ref{prop: separability R over R_Delta}
and Lemma \ref{lem:fixed}.

(b) follows from
Corollary~\ref{prop: sufficent conditions for separability},
Lemma \ref{lem:tracegammadelta}, and (a).

(c) follows from
Theorem \ref{prop: separability R over R_Delta, normal case},
Lemma \ref{lem:tracegammadelta}, and (a).

(d) follows from
Lemma \ref{lem:|T_f,e'|=|T_f,e|}
and (a), since $\alpha$ is trivial.

(e) follows from (c).

(f) follows from (d) with
$F_e: =\{ f \}$ and 
$b_f^e := 
[\Gamma(f):\Delta(f)]^{-1} \in B$,
$e \in \Gamma_0$.
\end{proof}

Let $I$ be a nonempty set.
Recall that $\Gamma$
is called the thin groupoid defined by $I$ if
$\Gamma$ has $I$ as its set of objects and
all pairs $(i,j): j \to i$,
for $i,j \in I$, as morphisms.
Composition of morphisms is defined by
$(i,j)(j,k) = (i,k)$,
$i,j,k \in I$. In that case, we write
$\Gamma = I \times I$.
Note that $\Gamma$ is connected.

\begin{corollary}\label{cor:thin}
Let $\Gamma = I \times I$ for some
nonempty set $I$.
Then $R/R_\Delta$ is separable.
\end{corollary}

\begin{proof}
For each $e \in \Gamma_0$,
the isotropy group $\Gamma(e)$ is trivial, and so $[\Gamma(e):\Delta(e)] = 1$.
Therefore, by Theorem~\ref{thm:crossed}(f),
$R/R_\Delta$ is separable.
\end{proof}

We end this section with an application
of Theorem \ref{thm:crossed} to a
construction from \cite{CLP,lundstrom2005}
of crossed products defined by
(not necessarily finite or Galois)
field extensions.

\begin{example}
Let $L/K$ be a separable field extension of
possibly infinite degree.
Let $\overline{K}$ denote a fixed algebraic
closure of $K$ containing $L$.
Choose a normal closure $N/K$ of $L/K$ in
$\overline{K}$.
Let $G$ denote the Galois group of $N/K$,
and let $\tilde{L} := (L_i)_{i \in I}$
be the family of distinct conjugate fields of $L$
under the action of $G$.

We now define a groupoid $\Gamma$ as follows.
We put $\Gamma_0 := \tilde{L}$.
If $i,j \in I$, then we put
$\Gamma(L_i,L_j) :=
\{ g|_{L_i} \mid g \in G,\ g(L_i)=L_j \}$.
We finally set
$\Gamma := \uplus_{i,j \in I} \Gamma(L_i,L_j)$.
Then $\Gamma$ is a connected groupoid.
For each $\sigma \in \Gamma$, we put
$\alpha_\sigma := \sigma$.
For every $(\sigma,\tau) \in \Gamma_2$,
choose $\beta_{\sigma,\tau} \in
L_{r(\sigma)} \setminus \{ 0 \}$
satisfying (O2) and (O4), as in
\cite{CLP,lundstrom2005}.
Then (O1) is automatic, and (O3) holds since the rings
$L_i$ are commutative fields and
$\alpha_\sigma \circ \alpha_\tau = \alpha_{\sigma\tau}$.
We say that
$\tilde{L} \rtimes^{\alpha}_{\beta} \Gamma$
is the object crossed product
defined by $L/K$ and $\beta$, and we
denote it by $(L/K,\beta)$.
We also put
$(L/K,\beta_\Delta) :=
\tilde{L}
\rtimes^{\alpha_\Delta}_{\beta_\Delta}
\Delta$.
\end{example}

\begin{theorem}\label{thm:relativecrossed}
With the above notation,
suppose that for each $j \in I$,
$\Delta(L_j)$ is closed in
$\Gamma(L_j)$ with respect to the
Krull topology.
Then $(L/K,\beta)/(L/K,\beta_\Delta)$
is separable if and only if there exists
$i \in I$ such that
$[\Gamma(L_i):\Delta(L_i)] < \infty$.
\end{theorem}

\begin{proof}
Suppose that
$(L/K,\beta)/(L/K,\beta_\Delta)$
is separable.
By Theorem \ref{thm:crossed}(a), there exists
$i \in I$ such that
$[\Gamma(L_i):\Delta(L_i)] < \infty$.

Conversely, suppose that there exists
$i \in I$ such that
$[\Gamma(L_i):\Delta(L_i)] < \infty$.
Since
$\Gamma(L_i)= {\rm Gal}(L_i/L_i^{\Gamma(L_i)})$
and $\Delta(L_i)$ is closed in $\Gamma(L_i)$, the
fundamental theorem of infinite Galois theory
(see \cite[Prop.~7.12 and Thm.~7.13(b)]{milne2022})
implies that
$
{\rm Gal} (L_i/L_i^{\Delta(L_i)}) 
= \Delta(L_i)$.
Because
$[\Gamma(L_i):\Delta(L_i)]<\infty$,
the subgroup $\Delta(L_i)$ is open in $\Gamma(L_i)$, and hence
$[L_i^{\Delta(L_i)}:L_i^{\Gamma(L_i)}]
=
[\Gamma(L_i):\Delta(L_i)]$.
Thus
$L_i^{\Delta(L_i)} / L_i^{\Gamma(L_i)}$
is a finite separable extension.
By \cite[Thm.~(4.6)]{reiner1975}, the usual field trace
$\operatorname{Tr}_{L_i^{\Delta(L_i)} / L_i^{\Gamma(L_i)}} :
L_i^{\Delta(L_i)} \to L_i^{\Gamma(L_i)}$
is surjective.
Moreover, by \cite[Rem.~5.47, p.~84]{milne2022}, this trace is given by the sum of the conjugates under
$
{\rm Gal} 
(L_i^{\Delta(L_i)}/L_i^{\Gamma(L_i)})
\cong \Gamma(L_i)/\Delta(L_i),
$
and therefore coincides with the relative trace map used in
Theorem \ref{thm:crossed}(b).
Hence Theorem \ref{thm:crossed}(b) implies that $(L/K,\beta)/(L/K,\beta_\Delta)$
is separable.
\end{proof}

\begin{corollary}[{\cite[Thm.~3]{CLP}}]
With the above notation,
suppose that $\Delta(L_i)=\{\id_{L_i}\}$
for all $i\in I$.
Then $(L/K,\beta)/(L/K,\beta_\Delta)$
is separable if and only if
$\Aut_K(L)$ is finite.
\end{corollary}

\begin{proof}
Since $\Delta(L_i)=\{\id_{L_i}\}$ for all $i\in I$,
these subgroups are closed in $\Gamma(L_i)$ with respect to the
Krull topology. Hence Theorem~\ref{thm:relativecrossed} implies that
$(L/K,\beta)/(L/K,\beta_\Delta)$ is separable if and only if there exists
$i\in I$ such that $|\Gamma(L_i)|<\infty$.

Now let $i,j\in I$. Since $L_i$ and $L_j$ are conjugate fields, there is
$g\in G = {\rm Gal}(N/K)$ with $g(L_i)=L_j$.
Conjugation by $g$ defines a group isomorphism
$\Gamma(L_i)\to \Gamma(L_j)$,
$\sigma\mapsto g\sigma g^{-1}$.
Thus all the groups $\Gamma(L_i)$, $i\in I$, are isomorphic.
In particular, $|\Gamma(L_i)|<\infty$ for some $i\in I$ if and only if
$|\Gamma(L_j)|<\infty$ for every $j\in I$.
Finally, choosing $j\in I$ such that $L_j=L$, we get $\Gamma(L_j)=\Aut_K(L)$.
Therefore, $(L/K,\beta)/(L/K,\beta_\Delta)$ is separable if and only if
$\Aut_K(L)$ is finite.
\end{proof}

\end{document}